\documentclass[11pt,a4paper]{article}%
\usepackage[centertags]{amsmath}
\usepackage{amsfonts}
\usepackage{amssymb}
\usepackage{amsthm}
\usepackage{epsfig}
\usepackage{setspace}
\usepackage{ae}
\usepackage{eucal}
\usepackage[usenames]{color}%
\usepackage{graphicx}

\theoremstyle{plain}
\newtheorem{thm}{Theorem}[section]

\newtheorem{prop}[thm]{Proposition}
\theoremstyle{definition}

\theoremstyle{remark}

\newtheorem{rem}[thm]{Remark}

\begin{document}

\title{On local analytic expansions of the densities in the context of (micro)-hypoelliptic and classes of semi-elliptic equations.}
\author{J\"org Kampen }
\maketitle

\begin{abstract}
Explicit representations of densities for linear parabolic partial differential equations are useful in order to design computation schemes of high accuracy for a considerable class of diffusion models. Approximations of lower order based on the WKB-expansion have been used in order to compute Greeks in standard models of the interest rate market (cf. \cite{KKS}). However, it turns out that for higher order approximations another related expansion leads to more accurate schemes. We compute a local explicit formula for a class of parabolic problems and determine a lower bound of the time horizon where it holds (given a certain bounded domain). Although the local analytic expansions hold only for strictly elliptic equations we show that the expansions can be used in order to design higher order schemes for various types of (micro)-hypoelliptic and semi-elliptic equations, e.g. the reduced market models considered in \cite{FKSE} or front fixing schemes for multivariate American derivatives \cite{K}.
\end{abstract}

\footnotetext[1]{Weierstrass Institute for Applied Analysis and Stochastics,
Mohrenstr. 39, 10117 Berlin, Germany.
\texttt{{kampen@wias-berlin.de}}.}
\section{Introduction}
Higher order approximations of fundamental solutions or densities for scalar linear partial differential equations of parabolic type are very useful as an element for the design of higher order schemes and efficient algorithms in the context of a large class of diffusion models. Note that this concerns also models which do not satisfy the rather strict conditions which we need in order to get local convergent expansions of densities for such equations. As examples consider the front-fixing  iteration for American options in \cite{K} or the analytic AD-scheme in \cite{FKSE} for semi-elliptic equations of type
\begin{equation}
	\label{projectiveHoermanderSystem0}
	\left\lbrace \begin{array}{ll}
		\frac{\partial u}{\partial t}=\frac{1}{2}\sum_{i=1}^mA_i^2u+A_0u\\
		\\
		u(0,x)=f(x),
	\end{array}\right.
\end{equation}
on the domain $[0,T]\times {\mathbb R}^n$, and where   
\begin{equation}
A_i=\sum_{j=1}^n a_{ji}\frac{\partial}{\partial x_j},
\end{equation}
and
are smooth vector fields vector fields which satisfy satisfy the H\"{o}rmander condition only with respect to the subspace of dimension $d\leq n$ at each point $x$, i.e. the sets 
\begin{equation}\label{Hoer0}
\left\lbrace A_i, \left[A_j,A_k \right], \left[ \left[A_j,A_k \right], A_l\right],\cdots |1\leq i\leq m,~0\leq j,k,l\cdots \leq m \right\rbrace 
\end{equation}
span a linear subspace $W_x$ of dimension $d\leq n$ at each point $x\in {\mathbb R}^n$. If $d=n$ then the equation  (\ref{projectiveHoermanderSystem0}) is micro-hypoelliptic and therefore hypoelliptic. Here $\left[.,.\right]$ denotes the Lie bracket of vector fields. Furthermore, recall that a differential operator $L$ with $C^{\infty}$-coefficients is called hypoelliptic on an open set $\Omega\subseteq {\mathbb R}^n$ if $Lu\in C^{\infty}$ implies $u$ in $C^{\infty}$ for any distribution $u$. Operators of the type of equation (\ref{projectiveHoermanderSystem0}) such that (\ref{Hoer0}) spans the full space ${\mathbb R}^n$ for all $x\in {\mathbb R}^n$ , are microhypoelliptic, i.e. preserve the wave front set 
\begin{equation}
\mbox{WF}(u)=\mbox{WF}(Lu),
\end{equation}
where $\mbox{WF}(u)$ is the intersection of the characteristic varieties of pseudo-differential operators $P$ of order zero which satisfy $Pu\in C^{\infty}$. Well, it is clear that the operators of (\ref{projectiveHoermanderSystem0}) are not microhypoelliptic for $d<n$ in general. This is reflected by the fact for pseudodifferential operators of negative order the characteristic variety equal the whole co-tangential bundle $T^0\Omega$ of the domain $\Omega$. Hence, the microhypoelliptic theory is designed for pseudodifferential operators of order $0$ (at least), and this does not apply in general if the H\"{o}rmander condition is satisfied on a subspace of lower dimension than $n$.

Note that typical diffusion models in finance, e.g. (multivariate versions) of the Heston model,  are often  microhypoelliptic or  hypoelliptic, but not strictly elliptic. Moreover, in many practical situations of related stochastic ordinary differential equations the number of Brownian motions may be less than than the dimension of the domain. The situation is considered in \cite{FKSE}. Another example of semi-elliptic problems which are not hypoelliptic in general, are free boundary problems. Consider for example American derivatives in finance. 
Let ${\mathbb R}^m_+$ denote the set of $m$-tuples of strictly positive real numbers. Then for each starting point $(t,x,y)\in [0,T]\times O\subseteq {\mathbb R}^m_+\times {\mathbb R}^d$ a typical diffusion market model is described by the stochastic differential equation 
\begin{equation}\label{marketincomp}
\begin{array}{ll}
(S_t^{t,x,y},Y_t^ {t,y})=(x,y)\in O,\\
\\
\frac{dS_s^ {t,x,k}}{S_s^ {t,x,k}}=r(s, S_s^ {t,x})ds+\sum_{j=1}^n\sigma_{kj}(s,S_s^ {t,x},Y_s^{t,x})dW^j_s,\\
\\
dY_s^{t,y,l}=\nu_l(s, Y_s^ {t,x})ds+\sum_{j=1}^p\beta_{lj}(s,Y_s^{t,x})dW^j_s,
\end{array}
\end{equation} 
which has a unique strong solution, if the drift functions
\begin{equation}
\begin{array}{ll}
r :[0,T]\times {\mathbb R}^m\rightarrow {\mathbb R},\\
\\
\nu_l :[0,T]\times {\mathbb R}^d\rightarrow {\mathbb R},~~l=1,\cdots ,d,
\end{array}
\end{equation}
and the (volatility of) volatility functions  
\begin{equation}
\beta_{lj} :[0,T]\times {\mathbb R}^d\rightarrow {\mathbb R},~~l\in\{1,\cdots ,p\},~~j\in\{1,\cdots ,d\}
\end{equation}
satisfy certain Lipschitz conditions. In this context the value function $V(t,x,y)$ of an American option with payoff $\phi$ is given by
\begin{equation}\label{price}
V(t,x,y):=\sup_{\tau \in \mbox{Stop}_{[t,T]}} E_Q\left[\exp\left( -\int_t^ {\tau}r(s, S_t^{t,x,y})\right)\phi\left( \tau , S_{\tau}^{t,x,y} \right)  \right],
\end{equation}
where $\mbox{Stop}_{[t,T]}$ is the set of all stopping times with value in $[t,T]$. The related obstacle problem is semi-elliptic. Such a problem is (micro)-hypoelliptic only on a subspace in general. 
Next a natural question occurs:
can we use approximations of fundamental solutions of strictly elliptic equations, i.e. equations of form (\ref{projectiveHoermanderSystem0}) which satisfy a uniform ellipticity condition on the whole space, in order to design algorithm in situations of (micro)-hypoellipticity or semi-ellipticity? More precisely and more general, assuming regularity and at most linear growth of the coefficients, how may we use local expansions of the fundamental solutions $p$ of form
\begin{equation}\label{genexp}
p(t,x;0,y)=\frac{1}{\sqrt{4\pi t}^n}\exp\left(-\frac{ d_R^2(x,y)}{4t} \right)\left( \sum_{k=0}^{\infty} d_{k}(x,y)t^k\right)?  
\end{equation}
in the context of the semi-elliptic types of equations mentioned above?
\begin{rem}
Note that representations of the form (\ref{genexp}) hold in the time-homogeneous case. In the time-inhomogeneous case they are of the form
\begin{equation}\label{genexp2}
p(t,x;0,y)=\frac{1}{\sqrt{4\pi t}^n}\exp\left(-\frac{ d_R^2(t,x,y)}{4t} \right)\left( \sum_{k=0}^{\infty} d_{k}(t,x,y)t^k\right), 
\end{equation}
and they are computerized in a similar way (see below).
\end{rem}
Here, the $d_k$ are solutions of recursively defined linear partial differential equations of first order, and $d_R$ is a Riemannian metric defined by the line element 
\begin{equation}\label{line}
ds^2=\sum_{i,j=1}^n a^{*ij}(x)dx_idx_j.
\end{equation}
\begin{rem}
Note that expansions of type (\ref{genexp2}) are different from the WKB-expansion considered in \cite{KKS} and in \cite{K}.  
\end{rem}

Note that the coefficients $a^{*ij}$ denote the components of the inverse of the diffusion matrix $(a^*_{ij})$ where the latter coefficients are determined by the vector filed coefficients $a_{ji}$ in the usual way. The interest in such expansions is related to the observation that the first order equations which determine the coefficient functions $d_k$ are easier to solve than the original second order equations. Well, it is sometimes not that easy. In order to determine the line element (\ref{line}) we need to solve nonlinear eikonal equations of the
form
\begin{equation}\label{eikonal}
d^2_R=\frac{1}{4}\sum_{i,j=1}^n a^*_{ij}d^2_{R,x_i}d^2_{R,x_j}.
\end{equation}
Furthermore, we have to construct solutions of (\ref{eikonal}) in regular spaces since the higher derivatives of $d^2_R$ appear as coefficients (and are involved in source terms) of the equations for the higher order terms $d_k$ in order to get accurate data for these equations.
Note that the 'iff'-condition on the 'boundary', i.e. the condition $d^2_r(x,y)=0$ if and only if $x=y$, leads to a term of lowest order of form
\begin{equation}\label{leadorder}
\sum_{i,j=1}^na^{*ij}(y)(x_i-y_i)(x_j-y_j)
\end{equation}
for $d^2_R$. From that term it looks hopeless to search for analytic expansions of type (\ref{genexp}) in the context of (micro)-hypoelliptic models. In such cases the points of elliptic degeneracies of such models are related to singularities of the inverse $(a^{ij})$ which defines the line element. However, this does not mean that it is impossible to find good approximations of the density in the form (\ref{genexp}) in many cases. Let us consider an example in finance. Univariate or multivariate stochastic volatility diffusion models like (\ref{marketincomp}) are usually of form
\begin{equation}
dS=\mu(S)dt+\sigma(S,Y)dW
\end{equation}
with $\mu$ some drift vector, and $\sigma$ some dispersion matrix-valued function. Lets assume that elliptic degeneracies of $\sigma\sigma^T$ appear for some set of arguments of measure zero (this is true for most of the standard stochastic volatility diffusion models such as the Heston model and multivariate versions of it). Indeed Malliavin calculus tells us that for models satisfying (\ref{Hoer0}) we typically have related weakly invertible covariance matrices. More precisely, if  
$n=d$ and the H\"{o}rmander condition holds, consider the 
associated Stratonovic integral of a process starting at $x\in {\mathbb R}^n$ is
\begin{equation}\label{strat}
X_t=x+\int_0^tA_0(X_s)ds +\sum_{k=1}^m \int_0^tA_k(X_s)\circ dW_k(s).
\end{equation}
Then the associated covariance matrix process $\sigma_t$ of the process satisfies
\begin{equation}
\sigma^{-1}_t\in L^p,
\end{equation}
where
\begin{equation}
\sigma_t=Z_t^{-1}\left[\int_0^tZ_sA_s(X_s)A^T_s(X_s)Z_s^T \right] Z_t^{-1,T},
\end{equation}
and where $Z$ is a matrix-valued invertible process defined by
\begin{equation}
Z_t=I_d-\int_0^t Z_sDA_0(X_s)ds-\sum_{i=1}^n\int_0^tZ_s DA_i(X_s)d\circ W_i(s).
\end{equation}
Here, $I_d$ denotes the $d$-dimensional identity matrix.
This indicates that the set of degeneracies is rather thin (of Lebesgue-measure zero). This is indeed what we observe in finance usually (cf. the Heston model or the SABR-model).
This leads us to the following consideration.
Let $D^{\epsilon}_{\mbox{{\tiny Deg}}}$ be the set of arguments $S,Y$ where the lower ellipticity constant of $\sigma\sigma^T(S,Y)$ is less or equal to $\epsilon$. Since the leading order term of (\ref{genexp}), i.e.
\begin{equation}\label{lead}
 \frac{1}{\sqrt{4\pi t}^n}\exp\left(-\frac{d_R^2}{2t}\right),
\end{equation}
goes rapidly to zero for arguments in $D^{\epsilon}_{\mbox{{\tiny Deg}}}$ if epsilon is small one may consider approximations $p_{\epsilon}$ of the density $p$ which are defined to be zero on the set $D^{\epsilon}_{\mbox{{\tiny Deg}}}$ and equal the expansion for some strictly elliptic operator with ellipticity constant $\epsilon>0$ in the complementary domain. Well, one has to control the time parameter $t$ to be not too large such that (\ref{lead})  really dominates the higher order terms $d_k$ where $d^2_R$ and derivatives of $d^2_R$ are involved. Then using the semigroup property one can set up weak higher order schemes (as considered in \cite{KKS} for example) which are time-discretized according to the nature of the degeneracies. The analysis of the time-discretization may be complicated sometimes, but it is a possible way of approximation in many situations of hypoelliptic operators considered by practitioners.    

From the computational point of view there are challenges other than degeneracies. Let us consider a few. First, the WKB-expansion 
\begin{equation}\label{wkbexp}
p(t,x,y)=\frac{1}{\sqrt{4\pi t}^n}\exp\left(-\frac{d_R^2(x,y)}{4t}  +\sum_{k=0}^{\infty} c_{k}(x,y)t^k\right),  
\end{equation}
considered in \cite{KKS} can lead to numerical instabilities if higher order terms $c_k$ with $k\geq 2$ are considered (see discussion below). In \cite{KKS} the expansion was considered up to the first order term $c_1$. The accuracy and efficiency of schemes based on analytic expansions is demonstrated by the fact that options in the Libor market with ten years of maturity are computable in a typical market situation in one time step. However, for higher volatilities higher order terms are desirable, and in that case numerical instabilities may appear. In order to have a damping leading order term as in (\ref{genexp}) we  compute the equations for $d_k$ instead of the $c_k$ in (\ref{wkbexp}). The resulting first order equations for $d_k$ are more difficult to solve than those for $c_k$. In this paper we compute recursive solutions for the $d_k$ in terms of the $c_k$, and recursively explicit solutions for the coefficients $d_k$ in the case where there is a global transformation of the second order part of the operator to the Laplacian (the 'reducible case'). In that case we also derive a lower bound for the radius of convergence. 
The use of first order solutions for the $d_k$ is not restricted to semi-elliptic equations with a thin set of elliptic degeneracies. For example high-dimensional models in finance are often reduced in order to get computationally feasible models. 
In \cite{FrKa} and in \cite{FKSE} reduced market models are considered which lead to semi-elliptic equations which are not micro-hypoelliptic (they are micro-hypoelliptic only on a subspace).
In general such equations do not have regular densities. Indeed they typically have densities only in a distributional sense.  For example, the lower dimensional Cauchy problem on $[0,T]\times{\mathbb R}^2$:
\begin{equation}
	\label{eq:simpleExampleReducedSystem}
	\left\lbrace \begin{array}{ll}
		\frac{\partial u}{\partial t}=\frac{1}{2}\sigma^2\frac{\partial^2 u}{\partial x_1^2}+\mu\frac{\partial u}{\partial x_2},\\
		\\
		u(0,x)=f(x_1)+g(x_2).
	\end{array}\right.
\end{equation}
has a `distributional density' of the form
\begin{equation}
p(t,x,y):=\frac{1}{\sqrt{2\pi t}^2}\exp\left(-\frac{(x_1-y_1)^2}{2\sigma^2t} \right)\delta(x_2+\mu t-y_2).
\end{equation}
In \cite{FKSE} analytical AD-schemes were defined based on analytical expansions of densities considered in this paper. We shall review these schemes below and show how analytic density approximations of this paper can be used.
In any case, the really hard part is the computation of the Riemannian metric and its derivatives. One approach based on regular polynomial interpolation is considered in \cite{K2}. However, we shall see that a more efficient method exists based on analytic schemes which we shall consider in a subsequent paper (second part of this work). The linear first order equation for the higher order terms are sometimes explicitly solvable, especially in the reducible case where a global transformation of the second order terms to the Laplacian exists. In general they may be computed by the regular interpolation method considered in \cite{K3}. 
Our first observation in this paper is that we can obtain recursive expressions for the $d_k$ from recursive expressions for the $c_k$ of the WKB-expansion, i.e.
in order to compute a local solution of
\begin{equation}\label{parabol1}
\frac{\partial u}{\partial t}= \sum_{ij=1}^na^*_{ij} \frac{\partial^2 u}{\partial x_i\partial x_j}
+\sum_i b_i(t,x)\frac{\partial u}{\partial x_i}
\end{equation}
(with analytic data) in the form 
\begin{equation}\label{ananewgen}
p(t,x;0,y)=\frac{1}{\sqrt{4\pi t}^n}\exp\left(-\frac{d^2_R}{4t} \right)\left( \sum_{k=0}^{\infty} d_{k}(t,x,y)t^k\right).  
\end{equation}
we may first compute the WKB-expansion
\begin{equation}\label{WKBgen}
p(t,x;0,y)=\frac{1}{\sqrt{4\pi t}^n}\exp\left(-\frac{d^2_R}{4t} +\sum_{k=0}^{\infty} c_{k}(t,x,y)t^k \right).  
\end{equation}
Indeed the connection between the $c_k$ and the $d_k$ can be easily computed using Leibniz rule. We have
\begin{equation}
 d_0=\exp(c_0),
\end{equation}
and
\begin{equation}
 d_k=\sum_{i=1}^k \frac{i}{k}d_{k-i}c_i.
\end{equation}
The case where the diffusion part can be globally transformed to a Laplacian is of special interest, because this makes it easier to study the convergence behavior of the higher order terms $d_k$ and the horizon of convergence. We call this case the reducible case.
Accordingly, in the first part of this paper we derive explicit analytic formulas of fundamental solutions to scalar linear equations of the form
\begin{equation}\label{parabol1}
\frac{\partial u}{\partial t}= \Delta u
+\sum_i b_i(t,x)\frac{\partial u}{\partial x_i}
\end{equation}
in terms of analytical representations of the coefficient functions $b_i$ and on a domain $D=(0,T]\times \Omega$ with $\Omega \subseteq {\mathbb R}^n$ a bounded domain. The expansion is local in time (as is the WKB-expansion).  Our explicit expansion is derived from the ansatz
\begin{equation}\label{ananew}
p(t,x,y)=\frac{1}{\sqrt{4\pi t}^n}\exp\left(-\frac{\sum_{i=1}^n \Delta x_i^2}{4t} \right)\left( \sum_{k=0}^{\infty} d_{k}(t,x,y)t^k\right).  
\end{equation}
Here, $\Delta x_i :=(x_i-y_i)$.
Note again that this is different from the WKB-expansion which is of the form
\begin{equation}\label{WKB}
p(t,x,y)=\frac{1}{\sqrt{4\pi t}^n}\exp\left(-\frac{\sum_{i=1}^n \Delta x_i^2}{4t}+\sum_{k=0}^{\infty} c_{k}(t,x,y)t^k \right).  
\end{equation}
The coefficients $d_k$ in (\ref{ananew}) are more difficult to compute than the coefficients
$c_k$ in (\ref{WKB}). However, concerning the growth of the coefficients with respect to the spatial variables we expect for coefficients with bounded derivatives and fixed time that
\begin{equation}
c_k\sim \Delta x^{2k}
\end{equation}
holds formally for the WKB expansion and this implies that higher order approximations involving $c_k$ for $k\geq 2$ may cause problems (we have no negative sign for $c_k$ on the whole domain in general). 
\begin{rem}
In the situations of low volatilities as considered in \cite{KKS} it is sufficient to compute the terms $c_0$ and $c_1$ even in order to get accurate results options with ten years maturity in a scheme with one time step. However, in general we need higher order approximations.
\end{rem}
Hence we expect for coefficients with bounded derivatives and for fixed time that 
\begin{equation}
d_k\sim \exp(c_0)\Delta x^{2k}
\end{equation}  
holds formally for expansions of the form (\ref{ananewgen}). However, as the growth of $c_0$ is linear in $\Delta x$ the highest order term of (\ref{ananewgen}) is an effective damping factor as $|\Delta x|$ becomes large (note, however, that we consider only bounded domains).  Expansions of the form (\ref{ananewgen}) were considered in a more general framework in proofs of the Atiyah-Singer index theorem of course. However, in that case only the behavior in the limit $t\downarrow 0$ is of interest. However, from a perspective of computational implementation it is also of interest to consider how large the time horizon can be chosen given a certain (bounded) domain and a certain set of coefficients $b_i$ such that an expansion of the form (\ref{ananewgen}) holds. First we do the analysis in the special case of (\ref{parabol1}). We shall assume that for all $1\leq i\leq n$ the functions $(t,x)\rightarrow b_{i}(t,x)$ are of linear growth and equal their Taylor expansion and have bounded derivatives of polynomial growth, i.e.
\begin{equation}\label{coeffb}
|D^n_tD^{\alpha}_x b_i(t,x)|\leq C^{n+|\alpha|}
\end{equation}
for some $C>0$. 
 Since we consider bounded domains and from the perspective of computation an important case are finite Fourier series, i.e. 
\begin{equation}\label{finfou}
b_i(t,x)=\sum_{j=-m_0}^{m_0}a_j\sin (j^0t+j\cdot x)+b_j\cos(j^0t+j\cdot x),
\end{equation}
where $m_0=(m_0^0,m_0^1,\cdots ,m_0^n)$ and $j=(j^0,j^1,\cdots ,j^n)$ are $n+1$-tuples, and $j\cdot x:=\sum_{i=0}^n j_i x_i$ denotes the scalar product of dimension $n$ where we identify $x_0$ with time $t$ for simplicity of notation. Since all $L^2$-functions can be approximated on a bounded domain up to any degree of accuracy the class of finite Fourier series is satisfying from a computational or practical  point of view.
In the more general case of variable diffusion coefficients $a^*_{ij}$ we assume that 
The result is also fundamental for investigation of the more general situation of equations with spatially dependent coefficients
\begin{equation}\label{parabol2}
\frac{\partial u}{\partial t}= \sum_{i,j=1}^na^*_{ij}(x)\frac{\partial^2 u}{\partial x_i\partial x_j}
+\sum_i b_i(x)\frac{\partial u}{\partial x_i},
\end{equation}
or  equations with coefficients dependent of space and time as in
\begin{equation}\label{parabol2}
\frac{\partial u}{\partial t}= \sum_{i,j=1}^na^*_{ij}(t,x)\frac{\partial^2 u}{\partial x_i\partial x_j}
+\sum_i b_i(t,x)\frac{\partial u}{\partial x_i},
\end{equation}
because some arguments will of the reducible case will transfer to the irreducible case. However, the extension hinges on a deeper analysis of (\ref{eikonal}). We will state the result but an extension of a the proof is given in a subsequent paper.
Since second order PDEs are essentially symmetric, it is essentially an equation of form
\begin{equation}\label{eikonal2}
d^2_R=\frac{1}{4}\sum_{i=1}^n \lambda_{i}(x)d^2_{R,x_i}d^2_{R,x_i},
\end{equation}
which has to be solved. For an extension of the convergence proof additional assumptions on the coefficient functions are needed. In order to compute a lower order bound of convergence we assume that
\begin{equation}\label{coeffa}
|D^n_tD^{\alpha}_x a_{ij}(t,x)|\leq C^{n+|\alpha|}
\end{equation}
for some $C>0$ and for all $1\leq i,j\leq n$. 
 Since we consider bounded domains and from the perspective of computation an important case are finite Fourier series, i.e. 
\begin{equation}\label{finfoua}
a_{ij}(t,x)=\sum_{j=-m_0}^{m_0}a^*_j\sin (j^0t+j\cdot x)+b^*_j\cos(j^0t+j\cdot x),
\end{equation}
with an analogous notation as above.
In the following we state a main theorem which makes the preceding remarks precise and then prove the theorem. Then in section 3 we compute a recursively explicit formulas for the coefficients $d_k$. In section 4 we prove convergence and determine a lower bound for the time horizon where the expansion converges.  In section 5 we consider extensions to parabolic equations with time-dependent coefficients. In section 6 we consider the expansion of parabolic equations with variable coefficients. 

In section 7 we consider application to (micro)-hypoelliptic and semi-elliptic equations, i.e. the design of weak higher order schemes in this context. In section 8 we consider applications to American derivatives as a second type of semi-elliptic equations which are not (micro)-hypoelliptic in general.

\section{A convergence result for reducible diffusion equations}
While heat-kernel expansions of the form (\ref{ananew}) are well-known the following questions deserve investigation:
\begin{itemize}
 \item what is a lower bound for the exact time horizon $T_0$ where such an expansion holds ? Can $T_0$ be computed in terms of the coefficient functions $b_i$? 
 
 \item is there an exact formula in terms of analytical expansions of the coefficient functions? 
\end{itemize}
Both questions are very important as they are fundamental in order to compute an efficient scheme for parabolic problems of type (\ref{parabol1}). The answer to the first questions tells us what time step size we have to choose in order to construct a locally analytic weak higher order scheme. The answer to the second question provides a formula for each time iteration step of the scheme, where the semi-group property is invoked to get a global scheme. 

We have  
\begin{thm}
Given assumption (\ref{coeffb}) there exists a finite time horizon $T_0$ such that on the domain $\Omega\times (0,T_0]$ for any finite $T_0>0$ and any domain $\Omega\subseteq {\mathbb R}^n$ a constant $\beta$ can be computed such that the fundamental solution of 
\begin{equation}\label{parasystthm}
\frac{\partial u}{\partial t}=\sum_{j=1}^n \frac{\partial^2 u}{\partial x_j^2} 
+\sum_{i=1}^n b_{i}\frac{\partial u}{\partial x_i}
\end{equation}
has the pointwise valid representation
\begin{equation}\label{pj}
p(t,x,0,y)=\frac{1}{\sqrt{4\pi t(\tau)}^n}\exp\left(-\frac{\sum_{i=1}^n \Delta x_i^2}{4t}\right)\left(\sum_{k=0}^{\infty}d_{k}(t,x,y)t^k \right),  
\end{equation}
for $j=1,\cdots ,n$, and for $(t,x) \in(0,\beta T_0)\times \Omega$. If (\ref{finfou}) holds then a lower bound of the constant $\beta$ is given by
\begin{equation}\label{beta0}
\beta <\frac{1}{3(n(2|m_0|+1))\overline{e}R^2|m_0|^{2}},
\end{equation}
where $2|m_0|+1$ is (an upper bound of) the number of terms in the finite Fourier representation of $b_i$ (any $i\in \left\lbrace 1,\cdots ,n\right\rbrace$ along with $|m_0|:=\max_{j\in \left\lbrace 0,\cdots n\right\rbrace}m^j_0$ and $R$ is a radius of a ball $B_R(0)$ such that the spatial part of the domain $\Omega$ is included, i.e. $B_R(0)\supseteq \Omega$, and $\overline{e}$ is an upper bound for the Fourier coefficients of the Fourier representation of the drift function $b_i$ where $i\in \left\lbrace 1,\cdots ,n\right\rbrace$.  
For the coefficient functions $d_k$ the following holds: for $k=0$ we have
\begin{equation}\label{c0}
 d_{0}(t,x,y)= \exp\left( \sum_m (y_m-x_m)\int_0^1 b_m (t,y+s(x-y))ds\right) ,
\end{equation}

\begin{equation}
d_m(t,x,y)=\sum_{k=1}^m \frac{k}{m}d_{m-k}\int_0^1 R_{k-1}(t,y+s(x-y),y)s^{k}ds
\end{equation}
with 
\begin{equation}\label{tk}
\begin{array}{ll}
R_{k-1}(t,x,y)=&\frac{\partial}{\partial t}c_{k-1}+\Delta c_{k-1}+\sum_{l=1}^n\sum_{r=0}^{k-1}\left( \frac{\partial}{\partial x_l}c_{r}\frac{\partial}{\partial x_l}c_{k-1-r}\right)\\
\\
&+\sum_{i} b_i(x)\frac{\partial}{\partial x_i}c_{k-1}
\end{array}
\end{equation}
\end{thm}

\begin{rem}
Note that the relation (\ref{beta0}) contains also the relation of the time horizon of the convergence to the size of the domain, i.e. the time horizon is proportional to the inverse of the square of the domain. 
\end{rem}

\begin{thm}
More explicitly, we have
\begin{equation}\label{powerco}
\begin{array}{ll}
 c_{0}(t,x,y)=c_0(x,y)=&-\sum_{i} \sum_{\gamma}b_{i\gamma}(y)\Delta x^{\gamma+1_i}\frac{1}{1+|\gamma|}\\
\\
&\equiv \sum_{\gamma} c_{0\gamma}\Delta x^{\gamma}
\end{array}
\end{equation}
and, given the power series representation 
\begin{equation}\label{powbeta}
c_{k-1}(t,x,y)
=\sum_{\gamma,l}c_{(k-1)\gamma l}(y)\Delta x^{\gamma}t^l
\end{equation}
we have
\begin{equation}\label{powerck}
\begin{array}{ll}
 c_{k} (t, x,y) =\sum_{\gamma,l}lc_{(k-1)\gamma l}(y)\Delta x^{\gamma}t^l+\\
\\
 \sum_{\gamma}{\big \{}\sum_i\sum_{\rho + \alpha = \gamma} (\rho_i +1)(\alpha_i +1 )c_{r(\beta +1_i)}c_{(k-1-r)(\alpha +1_i)}   \\
\\
 +\sum_i (\gamma_i+2)(\gamma_i+1)c_{k(\gamma+2_i)}+ 
\sum_{\rho + \alpha = \gamma}(\sum\frac{1}{\beta !}b_{i}(y)\times\\
\\
(\alpha_i+1)c_{(k-1)(\alpha+1_i)}{\big \}}\left( \sum_{\delta=0}^{\gamma}p_{k\delta}^{y\gamma} \Delta x^{\delta}\right),  
\end{array}
\end{equation}
where with $\displaystyle \delta_{\Sigma} := \sum_{i=1}^n \delta_i$, and
\begin{eqnarray}\label{abinom}
\sum_{\delta=0}^{\gamma}p_{k\delta, \beta,\tau}^{y\gamma} \Delta x^{\delta}&=&\sum_{\delta=0}^{\gamma}\frac{\beta}{(1-\tau)\delta_{\Sigma} +k}\nonumber\\
\\
&\times&\left[  \prod_{i=1}^n\left( \frac{\gamma_i !}{\delta_i ! (\gamma_i - \delta_i)!}\right)  y^{(\gamma - \delta)}\right] \Delta x^{\delta}.\nonumber
\end{eqnarray}
\end{thm}

\section{Formal computation of solution of the parabolic equation}
First we consider the equation (\ref{parabol1})  with time-homogeneous coefficients, i.e. where the coefficient functions $x\rightarrow b_{i}(x)$ depend only on the spatial variable $x$. It turns out that the ansatz
\begin{equation}\label{ansatzd}
p(t,x,y)=\frac{1}{\sqrt{4\pi t}^n}\exp\left(-\frac{\sum_{i=1}^n \Delta x_i^2}{4t}\right) \left(\sum_{k=0}^{\infty} d_k(x,y)t^k \right)  
\end{equation}
leads to first order partial differetial equations for $d_k$ with variable coefficients which are difficult to solve if looked at from an abstract point of view. Hence we compute the $d_k$ via the $c_k$ of the
WKB- ansatz
\begin{equation}\label{ansatzc}
p(t,x,y)=\frac{1}{\sqrt{4\pi t}^n}\exp\left(-\frac{\sum_{i=1}^n \Delta x_i^2}{4t}+\sum_{k=0}^{\infty} c_k(x,y)t^k \right).  
\end{equation}
Hence we first derive recursive relations for the coefficients $c_k$ and then get the recursive relations for the $d_k$ via the general logarithmic recursion outlined in the introduction, i.e. we use the recursion
\begin{equation}
 d_0=\exp(c_0),
\end{equation}
and
\begin{equation}
 d_k=\sum_{i=1}^k \frac{i}{k}d_{k-i}c_i.
\end{equation}
In a second step, using global analyticity of the coefficient functions $b_{i}$, we derive the explicit solution in terms of Taylor power series of $b_{i}$. 
For the time derivative we get
\begin{equation}
\frac{\partial p}{\partial t}(t,x)=\left(-\frac{n}{2t}+\frac{\sum_i \Delta x_i^2}{4t^2}+\sum_k kc^j_k(x,y)t^{k-1}\right)p (t,x,y).
\end{equation}
For the first and second spatial derivatives we get
\begin{equation}
\frac{\partial p}{\partial x_l}=\left(\frac{-\Delta x_l}{2t}+\sum_k \frac{\partial}{\partial x_l}c_k(x,y)t^k \right) p(t,x,y),
\end{equation}
and
\begin{equation}
\begin{array}{ll}
\frac{\partial^2 p}{\partial x_l^2}=&{\Bigg (}-\frac{1}{2t}+\sum_k \frac{\partial^2}{\partial x_l^2}c_k(x,y)t^k \\
\\
&+\left(-\frac{\Delta x_l}{2t}+\sum_k \frac{\partial}{\partial x_l}c_k(x,y)t^k \right)^2{\Bigg )} p(t,x,y).
\end{array}
\end{equation}
Plugging into (\ref{parabol1}) and ordering with respect to the terms $t^{-2},t^{-1}$ etc. we get the following
recursive relations for the $c^j_k$, where $1\leq j\leq n$:
\begin{equation}\label{tsmminus2}
t^{-2}:~~\frac{\sum_i \Delta x_i^2}{4t^2}=\sum_{l}\frac{ \Delta x_l^2}{4t^2}
\end{equation}

\begin{equation}\label{tsmminus1}
t^{-1}:~~-\frac{n}{2t}=-\sum_{l}\frac{1}{2t}-\frac{1}{2t}\left( \sum_l \Delta x_l\frac{\partial c_0^j}{\partial x_l}-\sum_{lm} b^j_{lm}(x)\Delta x_m\right) , 
\end{equation}
and for all $k-1\geq 0$.
\begin{equation}\label{tsmk}
\begin{array}{ll}
t^{k-1}:~~kc_k+\sum_l \Delta x_l\frac{\partial c_k}{\partial x_l}=&\Delta c_{k-1}+\sum_{l=1}^n\sum_{r=0}^{k-1}\left( \frac{\partial}{\partial x_l}c_r\frac{\partial}{\partial x_l}c_{k-1-r}\right)\\
\\
&+\sum_{m} b_{m}(x)\frac{\partial}{\partial x_m}c_{k-1}\equiv R_{k-1}(x,y).
\end{array}
\end{equation}
Note that the first order coupling of the system is essentially reflected in the recursive first order partial differential equations starting from (\ref{tsmk}). This would be different if we had coupling via the second order terms and it makes the solution of the system much easier.
Note that equation (\ref{tsmminus2}) is satisfied. Equation (\ref{tsmminus1}) is equivalent to
\begin{equation}
\sum_l \Delta x_l\frac{\partial c_0}{\partial x_l}=-\sum_{m} b_{m}(x)\Delta x_m, 
\end{equation}
with the solution
\begin{equation}
 c_0(x,y)=\sum_m (y_m-x_m)\int_0^1 \sum_lb_{l}(y+s(x-y))ds
\end{equation}
and for all $k\geq 1$ we have
\begin{equation}
c_k(x,y)=\int_0^1 R_{k-1}(y+s(x-y),y)s^kds
\end{equation}
with $R_{k-1}$ as in equation \eqref{tk}.
Next we compute the solution explicitly doing the integral for $c_0$ first. We abbreviate $\Delta x=(x-y)$ with components $\Delta x_i=(x-y)_i$ and for a multiindex $\alpha=(\alpha_1,\cdots, \alpha_n)$ we write $\Delta x^{\alpha}:=\Pi_{i=1}^n\Delta x_i^{\alpha_i}$. Furthermore, we define $|\alpha|=\sum_i \alpha_i$ If
\begin{equation}
b_{m}(x)=\sum_{\gamma} \frac{1}{\gamma !}b_{m,\gamma}(y)(\Delta x)^{\gamma},
\end{equation}  
along with some multiindex $\gamma$, then
\begin{equation}
\begin{array}{ll}
 c_0(x,y)&=-\sum_m \Delta x_m \int_0^1 \sum_l b_{m}(y + s\Delta x)ds\\
\\
&=-\sum_{m} \Delta x_m \int_0^1 \sum_{\gamma} b_{m\gamma}(y)(s\Delta x)^{\gamma}ds\\
\\
&=-\sum_{l,m} \Delta x_m \sum_{\gamma} b_{m\gamma}(y)\Delta x^{\gamma} \int_0^1 s^{|\gamma|}ds\\
\\
&=-\sum_{l,m} \sum_{\gamma}b_{m\gamma}(y)\Delta x^{\gamma+1_i}\frac{1}{1+|\gamma|}s^{|\gamma|+1}\Big |_0^1\\
\\
&=-\sum_{l,m} \sum_{\gamma}b_{m\gamma}(y)\Delta x^{\gamma+1_i}\frac{1}{1+|\gamma|}\\
\\
&\equiv \sum_{\gamma} c_{0\gamma}\Delta x^{\gamma}.
\end{array}
\end{equation}
Next we compute $c_k$ for $k\geq 1$.  We have
\begin{equation}
\begin{array}{ll}
 c_{k} (x,y) =&\int_0^1 \{\sum_i \sum_{r=0}^{k-1} \dfrac{\partial c_r}{\partial x_i}\dfrac{\partial c_{k-1-r}}{\partial x_i}\\
\\
 &+ \Delta c_{k-1} + \sum_m b_{m} \dfrac{\partial c_{k-1}}{\partial x_m}\} (y+s(x-y))s^{k-1}ds.
\end{array}
\end{equation}
Assuming that $c_{k-1}$ equals its Taylor series for every $y\in {\mathbb R}^n$, i.e.
\begin{equation}
c_{k-1}(x)=\sum_{\gamma} c_{(k-1)\gamma}(y)\Delta x^{\gamma},
\end{equation}
then we may evaluate the derivatives occurring in $R_{k-1}$ as follows:
\begin{equation}
 \dfrac{\partial c^j_{k-1}}{\partial x_i}=\sum_{\gamma}(\gamma_i+1)c_{(k-1)(\gamma+1_i)}(y)\Delta x^{\gamma},
\end{equation}
and
\begin{equation}
\dfrac{\partial^2 c_{k-1}}{\partial x_i^2}=\sum_{\gamma}(\gamma_i+2)(\gamma_i+1)c_{k(\gamma+2_i)}(y)\Delta x^{\gamma},
\end{equation}
and
\begin{equation}
 \dfrac{\partial c_{r}}{\partial x_i}\dfrac{\partial c_{k-1-r}}{\partial x_i}=\sum_{\gamma}\left\lbrace \sum_{\beta + \alpha = \gamma} (\beta_i +1)(\alpha_i +1 )c_{r(\beta +1_i)}c_{(k-1-r)(\alpha +1_i)} \right\rbrace \Delta x^{\gamma}. 
\end{equation}
For the multiindex $\gamma$, we have
\begin{eqnarray}\label{abinom}
 P_k^{\gamma}(x,y) &:=&\int_0^1 (y + s(x-y))^{\gamma} s^{k-1}ds\\
&=&\int_0^1 \prod_{i=1}^n\left(\sum_{\delta_i =0}^{\gamma_i} \frac{\gamma_i !}{\delta_i ! (\gamma_i - \delta_i)!}y_i^{(\alpha_i - \delta_i)}\Delta x^{\delta_i} s^{\delta_i}\right)s^{k-1}ds\nonumber\\
&=&\int_0^1\sum_{\delta=0}^{\gamma}\left(  \prod_{i=1}^n\frac{\gamma_i !}{\delta_i ! (\alpha_i - \delta_i)!}y_i^{(\gamma_i - \delta_i)}\Delta x^{\delta_i}\right) s^{\delta} s^{k-1} ds \nonumber\\
&=&\sum_{\delta=0}^{\gamma}\frac{1}{\delta_{\Sigma} +k}\left[  \prod_{i=1}^n\left( \frac{\gamma_i !}{\delta_i ! (\gamma_i - \delta_i)!}\right)  y^{(\gamma - \delta)}\right] \Delta x^{\delta}\nonumber\\
&=:&\sum_{\delta=0}^{\gamma}p_{k\delta}^{y\gamma} \Delta x^{\delta}\nonumber
\end{eqnarray}
where $\displaystyle \delta_{\Sigma} := \sum_{i=1}^n \delta_i$ and $s^{\delta}=\Pi_{i=1}^{n}s^{\delta_i}=s^{\delta_{\Sigma}}$. Hence
\begin{equation}
\begin{array}{ll}
 c^j_{k} (x,y) =\\
\\
 \sum_{\gamma}{\big \{}\sum_i\sum_{\beta + \alpha = \gamma} (\beta_i +1)(\alpha_i +1 )c^j_{r(\beta +1_i)}c^j_{(k-1-r)(\alpha +1_i)}   \\
\\
 +\sum_i (\gamma_i+2)(\gamma_i+1)c_{k(\gamma+2_i)}+ 
\sum_{\beta + \alpha = \gamma}(\sum\frac{1}{\beta !}b^{j}_{lm,\beta}(y)\times\\
\\
(\alpha_i+1)c_{(k-1)(\alpha+1_i)}{\big \}}\left( \sum_{\delta=0}^{\gamma}p_{k\delta}^{y\gamma} \Delta x^{\delta}\right).  
\end{array}
\end{equation}

\section{Proof of convergence and computation of the time horizon of convergence}
It is essential to show that for fixed $x,y\in \Omega$ (where $\Omega$ is some bounded domain) 
\begin{equation}
t\rightarrow \sum_{k=0}^{\infty} c_k(x,y) t^k  
\end{equation}
represents some analytic function on some interval $ [0,\beta T_0]$ for some $\beta$ to be determined. Then it follows that
\begin{equation}
t\rightarrow  \sum_{k=0}^{\infty} d_k(x,y)t^k=\exp\left(\sum_{k=0}^{\infty} c_k(x,y) t^k \right) 
\end{equation}
is also an analytic function on that interval.

Since $\Omega\subset {\mathbb R}^n$ is bounded, there is a ball $B_R(0)$ around $0$ with radius $R$ such that $\Omega\subset B_R(0)$. Recall that
\begin{equation}\label{cj0}
 c_0(x,y)=\sum_m (y_m-x_m)\int_0^1 b_{m}(y+s(x-y))ds,
\end{equation}
and for all $k\geq 1$ we have
\begin{equation}
c_k(x,y)=\int_0^1 R_{k-1}(y+s(x-y),y)s^{k-1}ds
\end{equation}
with 
\begin{equation}\label{tk}
\begin{array}{ll}
R_{k-1}(t,x,y)=&\Delta c_{k-1}+\sum_{l=1}^n\sum_{r=0}^{k-1}\left( \frac{\partial}{\partial x_l}c_r\frac{\partial}{\partial x_l}c_{k-1-r}\right)\\
\\
&+\sum_{m} b_{m}(x)\frac{\partial}{\partial x_m}c_{k-1}.
\end{array}
\end{equation}
Let us first consider $c_0$. The coefficients of equation (\ref{finfou}) have a representation of the form

\begin{equation}\label{finfou2}
b_i(t,x)=\sum_{j=-m_0}^{m_0}e_j\exp (i(j^0t+j\cdot x))+f_j\exp(-i(j^0t+j\cdot x)),
\end{equation}
for some real numbers $e_j$ and $f_j$. Let
\begin{equation}
\overline{e}=\max\left\lbrace e_j,f_j|j\in J \right\rbrace 
\end{equation}
where $J$ is the set of all multiindices of the sum in (\ref{finfou2}). Then we get the estimate
\begin{equation}
\begin{array}{ll}
c_0(x,y)=\sum_{m=1}^n(y_m-x_m)\int_0^1 \sum_m b_{m}(y+s(x-y))ds\\
\\
\leq \overline{e}R(2m_0+1)=:C_0.
\end{array}
\end{equation} 
Next the time transformation
\begin{equation}
t=\beta \tau
\end{equation}
transforms the equation
\begin{equation}\label{parasyst3a}
\frac{\partial u}{\partial t}=\sum_{j=1}^n \frac{\partial^2 u}{\partial x_j^2} 
+\sum_{k} b_{k}\frac{\partial u}{\partial x_k}
\end{equation}
into the equation
\begin{equation}\label{parasyst3a} 
\frac{\partial v}{\partial \tau}=\beta\sum_{j=1}^n \frac{\partial^2 v}{\partial x_j^2} 
+\beta\sum_{k} b_{k}\frac{\partial v}{\partial x_k},
\end{equation}
where $u(t,x)=v(\tau,x)$, and where $\frac{\partial u}{\partial t}=\frac{\partial v}{\partial \tau}\frac{\partial \tau}{\partial t}=\frac{\partial v}{\partial \tau}\frac{1}{\beta}$.
The analogous representation of the solution is of the form
\begin{equation}\label{betapj}
p^{\beta}(t,x,y)=\frac{1}{\sqrt{4\pi \tau}^n}\exp\left(-\frac{\sum_{i=1}^n \Delta x_i^2}{4\beta\tau}\right)\left(1+\sum_k c_{k,\beta}(x,y)\tau^k \right),  
\end{equation}
for $j=1,\cdots ,n$. Plugging (\ref{betapj}) into (\ref{parasyst3a}) and collecting the terms with $\tau^{-2}$, $\tau^{-1}$ etc. we get (we feel free to write $t$ instead of $\beta \tau$ if convenient)
\begin{equation}\label{tauminus2}
\tau^{-2}:~~\frac{\sum_i \Delta x_i^2}{4\beta \tau^2}=\beta\sum_{l}\frac{ \Delta x_l^2}{4\beta^2 \tau^2}
\end{equation}

\begin{equation}\label{tauminus1}
\tau^{-1}:~~-\frac{n}{2t}=-\beta\sum_{l}\frac{1}{2\beta t}-\frac{\beta}{2\beta \tau}\left( \sum_l \Delta x_l\frac{\partial c_{0,\beta}}{\partial x_l}-\sum_{m} b_{m}(x)\Delta x_m\right) , 
\end{equation}
and for all $k-1\geq 0$
\begin{equation}\label{tk}
\begin{array}{ll}
\tau^{k-1}:~~kc_{k,\beta}+\beta\sum_l \Delta x_l\frac{\partial c_{k,\beta}}{\partial x_l}=&\beta\Delta c_{k-1,\beta}+\beta\sum_{l=1}^n\sum_{r=0}^{k-1}\left( \frac{\partial}{\partial x_l}c_{r,\beta}\frac{\partial}{\partial x_l}c_{k-1-r,\beta}\right)\\
\\
&+\beta\sum_{m} b_{m}(x)\frac{\partial}{\partial x_m}c_{k-1,\beta}=: R^{\beta}_{k-1}(x,y).
\end{array}
\end{equation} 
We divide equation (\ref{tk}) by $\beta$ and get the
solutions (the solution for $c_{0,\beta}$ equals exactly that for $c_0$ in (\ref{cj0}))
\begin{equation}\label{jbk}
c_{k,\beta}(x,y)=\int_0^1 R^{\beta}_{k-1}(y+s(x-y),y)s^{\frac{k}{\beta}-1}ds.
\end{equation}
Next we prove 
\begin{thm}
For a give time horizon $T_0>0$ there exists $\beta >0$ such that 
\begin{equation}
\mbox{for each}~~x,y\in \Omega, 1\leq j\leq n~~c_{k,\beta}(x,y)\downarrow 0 \mbox{ as } k\uparrow \infty.
\end{equation}
\end{thm}
\begin{proof}
Next, a majorant of $c_{k,\beta}(x,y)$ is obtained as follows: we consider three types of operators $O^{1,n}_k, O^{2,n}_k, O^{3,n}_k$ with positive integers $k$, and acting on a single function $f:\Omega\times \Omega\rightarrow {\mathbb R}$ or on a families of functions $(f_l)_{1\leq l\leq k}:\Omega\times \Omega\rightarrow {\mathbb R}$, namely
\begin{equation}
\begin{array}{ll}
O^{1,n}_k\left[f\right](x,y):=\frac{\beta}{k}\Delta f(x,y)\\
\\
O^{2,n}_k\left[f_k,\cdots,f_1\right](x,y):=\frac{\beta}{k}\sum_{l=1}^n\sum_{r=0}^k \frac{\partial f_r}{\partial x_l}\frac{\partial f_{k-r}}{\partial x_l}\\
\\
O^{3,n}_k\left[f\right](x,y):=\frac{\beta}{k}\sum_{m} b_{m}(x)\frac{\partial}{\partial x_m}f(x,y).
\end{array}
\end{equation}
Let
\begin{equation}
M_k:=\left\lbrace (\alpha_k,\cdots ,\alpha_1)|\alpha_j\in \left\lbrace 1,2,3\right\rbrace \right\rbrace 
\end{equation}
For 
\begin{equation}
c^{up}_{k,\beta}:=\sup_{x,y\in\Omega}c_{k,\beta}(x,y) 
\end{equation}
we have
\begin{equation}\label{boundest}
c_{k,\beta}^{up}\leq \sup_{x,y\in\Omega}\sum_{\alpha\in M_k}O_k^{\alpha ,n}c_0(x,y).
\end{equation}
where
\begin{equation}
O^{\alpha ,n}_k\left[f\right](x,y):=O^{\alpha_k,n}_kO^{\alpha_{k-1},n}_{k-1}\circ\cdots\circ O^{\alpha_1,n}_1\left[f\right](x,y). 
\end{equation}

First let ${\bf 1}_k$ (resp. ${\bf 2}_k, {\bf 3}_k$) the multiindex $\alpha \in M_k$ such that for each $1\leq m\leq k$ $\alpha_m=1$ (resp. $\alpha_m=2, \alpha_m=3$). Hence
\begin{equation}\label{abb}
\begin{array}{ll}
O^{{\bf 1}}_k\left[f\right] (x,y)=\left( O^1\right)^k\left[f\right] (x,y)=\Delta^k\left[f\right] (x,y) 
\end{array}
\end{equation}
etc.. In order to compute this expression we first apply the Laplacian to $\exp(m_0ix)$. We get
\begin{equation}\label{eqlaplace}
{\big |}\Delta^k\exp(im_0x){\big |}\leq  n^{k}|m_{0}|^{2k},
\end{equation}
where $|m_0|:=\max_{j\in\left\lbrace 0,1,\cdots ,n \right\rbrace} m_0^j$.
Since there are $n$ drift functions $b_i$ the explicit representation of $c_0$ has $n(2|m_0|+1)$ terms with the factor $\exp(i(j\cdot x))$. Since the dimension is $n$ we have $n$ factors $\delta x_i:=(x_i-y_i)$ in the definition of $c_0$. Note that for $m\geq 0$
\begin{equation}
\begin{array}{ll}
\left( \frac{\partial^2}{\partial x_i^2}\right)^m \left( \delta x_i\exp(i(j\cdot x))\right) \\
\\
=m\left( \frac{\partial^2}{\partial x_i^2}\right)^{m-1}\exp(i(j\cdot x))+\delta x_i\left( \frac{\partial^2}{\partial x_i^2}\right)^{m}\exp(i(j\cdot x)).
\end{array}
\end{equation}  
Hence, 
\begin{equation}
\frac{1}{k!}\sup_{x,y\in \Omega}{\big |}\Delta^kc_0(x,y){\big |}\leq
n(2|m_0|+1)
 \overline{e}(k+R)\frac{\beta^k \left( n^{k}|m_{0}|^{2k}\right)}{k!}\downarrow 0
\end{equation}
as $k\uparrow \infty$ if
\begin{equation}
0<\beta \leq \frac{1}{n |m_0|^2}.
\end{equation}
Next we observe that on a domain of radius 
\begin{equation}
|O^{{\bf 3,n}}_k\left[c_0(x,y)\right] (x,y)|\leq 
\overline{e}^kR\frac{\beta^k k \left( k!n^{k}(2|m_{0}|+1)^{k}\right)}{k!}\downarrow 0
\end{equation}
as $k\uparrow \infty$
 if (note the strict inequality sign because of additional factor $k$)
\begin{equation}
0<\beta < \frac{1}{\overline{e}(n(2|m_0|+1))}.
\end{equation}
The operators of quadratic type applied to $c_0(x,y)$ $O^{2,n}_{k}c_0$ decrease also to zero as $k\uparrow \infty$ if $\beta$ is small. Inductively with respect to $k$ you show that
\begin{equation}
|O^{{\bf 2},n}_k\left[\exp(im_0x)\right] (x,y)|\leq |m_0|^{2k}3^k k!.
\end{equation}
We get
\begin{equation}\label{large}
\frac{1}{k!}|O^{{\bf 2},n}_k\left[c_0\right] (x,y)|\leq  \frac{\beta^k (n(2|m_0|+1))^k\overline{e}^kR^{2k}|m_0|^{2k}3^k k!}{k!}\downarrow 0
\end{equation}
for $k\uparrow \infty$ if
\begin{equation}\label{req3}
\beta <\frac{1}{3(n(2|m_0|+1))\overline{e}R^2|m_0|^{2}}
\end{equation}

For large $k$ this is essentially the largest term of all the $3^k$ contributions in the sum (\ref{boundest}) for large $k$ ($k$ fixed). You can show that if (\ref{req3}) holds, then
we have for $k\geq k_0$ (some $k_0>0$
\begin{equation}\label{essest}
|3^kO^{{\bf 2}}_k\left[c_0\right] (x,y)|\leq \frac{3^k\beta^k k^2 2^{k-1}n^kC^{k}k!(c_0)^{k+1}}{k!}\downarrow 0.
\end{equation}
as $k\uparrow \infty$, and this is also the estimate which holds for $c_k$ for large $k$. Here we choose $\beta$ such that in a summand in $O_k^{\alpha ,n}c_0(x,y)$ in (\ref{boundest}) each occurrence of an operator of form $O^{3,n}_k$ can be replaced by an operator of form  $O^{2,n}_k$ in order to get a majorant estimation. So in the sum in (\ref{boundest}) it suffices to concentrate on the summands consisting of concatenations of operators of form $O^{2,n}_k$ and $O^{1,n}_k$. For natural numbers $l$ let us define an increasing sequence of numbers $k_1<k_2<\cdots <k_l<k_{l+1}\cdots$, and operators
\begin{equation}
\begin{array}{ll}
O^{1,n}_{k_{l+1}k_l}:=O^{1,n}_{k_{l+1}}\circ\cdots \circ O^{1,n}_{k_l}\\
\\
O^{2,n}_{k_{l+1}k_l}:=O^{2,n}_{k_{l+1}}\circ\cdots \circ O^{2,n}_{k_l}
\end{array}
\end{equation}
Then in the summands o (\ref{boundest}) we have to consider the asymptotic behavior of values of family of operators of form
\begin{equation}\label{fam1}
O^{2,n}_{k_{l+1}k_l}\circ O^{1,n}_{k_{l}k_{l-1}}\circ \cdots \circ O^{2,n}_{k_{3}k_2}\circ O^{1,n}_{k_{2}k_1}
\end{equation}
or of form
\begin{equation}\label{fam2}
O^{2,n}_{k_{l+1}k_l}\circ O^{1,n}_{k_{l}k_{l-1}}\circ \cdots \circ O^{1,n}_{k_{3}k_2}\circ O^{2,n}_{k_{2}k_1}
\end{equation}
applied to $c_0(x,y)$ as $k\uparrow \infty$. If there is only a finite occurrence of operators of form $O^{1,n}_k$ in such a family ((\ref{fam1}) or (\ref{fam2})), then the asymptotic behavior is clearly the same as for 
$O^{{\bf 2}}_k c_0(x,y)$. 
If on the other hand there are infinite occurrences of operators of form $O^{1,n}_k$ 
in ((\ref{fam1}) of (\ref{fam2})), then for large $k$ $O^{{\bf 2}}_k c_0(x,y)$ becomes a 
majorant of such a summand. Hence, the estimate (\ref{essest}) is a majorant for large $k$ and proves the convergence of the series in (\ref{boundest}). 
\end{proof}

\section{Generalisation to the time-inhomogeneous case }
We start with the formal computation in the time-inhomogeneous case and then show how the local convergence proof of the previous section can be extended.

\subsection{Formal computation of recursive coefficients in the time-inhomogeneous case}
We consider parabolic equations with time-dependent coefficients of the form
\begin{equation}
\frac{\partial u}{\partial t}+\Delta u +\sum_{k} b_{k}(t,x)\frac{\partial u}{\partial x_k}=0
\end{equation}
We consider the ansatz
\begin{equation}
p(t,x,0,y)=\frac{1}{\sqrt{4\pi t}^n}\exp\left(-\frac{\Delta x^2}{4t}+ \sum_{k=0}^{\infty}c_k(t,x,y)t^k\right).
\end{equation}
Compared to the time-homogeneous case the time derivative contains an additional term. We have
\begin{equation}
\begin{array}{ll}
\frac{\partial p}{\partial t}(t,x,0,y)=
{\Bigg (}-\frac{n}{2t}+\frac{\sum_i \Delta x_i^2}{4t^2}+\sum_{k=0}^{\infty}\frac{\partial c_k}{\partial t}(t,x,y)t^k\\
\\
\hspace{5.5cm}+\sum_k kc_k(t,x,y)t^{k-1}{\Bigg )}p (t,x,0,y)
\end{array}
\end{equation}
The spatial derivatives are essentially the same as in the time-homogeneous case. We compute
\begin{equation}
\frac{\partial p}{\partial x_l}(t,x,y)=\left(\frac{-\Delta x_l}{2t}+\sum_k \frac{\partial}{\partial x_l}c_k(t,x,y)t^k \right) p(t,x,0,y),
\end{equation}
and
\begin{equation}
\begin{array}{ll}
\frac{\partial^2 p}{\partial x_l^2}(t,x,y)=&{\Bigg (}-\frac{1}{2t}+\sum_k \frac{\partial^2}{\partial x_l^2}c_k(t,x,y)t^k \\
\\
&+\left(-\frac{\Delta x_l}{2t}+\sum_k \frac{\partial}{\partial x_l}c_k(t,x,y)t^k \right)^2{\Bigg )} p(t,x,0,y).
\end{array}
\end{equation}
Plugging into equation (\ref{parabol1}) and ordering with respect to the terms $t^{-2},t^{-1}$ etc. we get the following
recursive relations for the $c^j_k$, where $1\leq j\leq n$. First, the highest order terms are the same as before:
\begin{equation}\label{tminus2}
t^{-2}:~~\frac{\sum_i \Delta x_i^2}{4t^2}=\sum_{l}\frac{ \Delta x_l^2}{4t^2}
\end{equation}
The terms of order $t^{-1}$ are essentially as before (we just have to add the $t$-argument in the coefficient functions $b_{j}$):
\begin{equation}\label{tminus1}
t^{-1}:~~-\frac{n}{2t}=-\sum_{l}\frac{1}{2t}-\frac{1}{2t}\left( \sum_l \Delta x_l\frac{\partial c_0}{\partial x_l}-\sum_{m} b_{m}(t,x)\Delta x_m\right). 
\end{equation}
For $k-1\geq 0$ we get an additional $t$-derivative on the right side:
\begin{equation}\label{timetk}
\begin{array}{ll}
t^{k-1}:~~kc_k+\sum_l \Delta x_l\frac{\partial c_k}{\partial x_l}=\frac{\partial c_{k-1}}{\partial t}+\Delta c_{k-1}+\sum_{l=1}^n\sum_{r=0}^{k-1}\left( \frac{\partial}{\partial x_l}c_r\frac{\partial}{\partial x_l}c_{k-1-r}\right)\\
\\
+\sum_{m} b_{m}(t,x)\frac{\partial}{\partial x_m}c_{k-1}\equiv R_{k-1}(x,y)
\end{array}
\end{equation}
Hence,
\begin{equation}
\sum_l \Delta x_l\frac{\partial c_0}{\partial x_l}=-\sum_{m} b_{m}(t,x)\Delta x_m, 
\end{equation}
which has the solution
\begin{equation}
 c_0(x,y)=\sum_m (y_m-x_m)\int_0^1 \sum_lb_{l}(t,y+s(x-y))ds,
\end{equation}
and for all $k\geq 1$ we have
\begin{equation}
c_k(x,y)=\int_0^1 R_{k-1}(t,y+s(x-y),y)s^kds
\end{equation}
with $R_{k-1}$ as in equation \eqref{timetk}.
The explicit calculation of the solution is know completely analogous, so it suffices to write down the results.
We write
\begin{equation}
b_{m}(t,x)=\sum_{\gamma} \frac{1}{\gamma !}b_{m,\gamma}(t,y)(\Delta x)^{\gamma}
\end{equation}  
along with some multiindex $\gamma$. Then
\begin{equation}
\begin{array}{ll}
 c_0(t,x,y)&=-\sum_{m} \sum_{\gamma}b_{m\gamma}(y)\Delta x^{\gamma+1_i}\frac{1}{1+|\gamma|}\\
\\
&\equiv \sum_{\gamma} c_{0\gamma}(t,y)\Delta x^{\gamma}
\end{array}
\end{equation}
Given that $c_{k-1}$ equals its Taylor series for every $y\in {\mathbb R}^n$, i.e.
\begin{equation}
c_{k-1}(t,x)=\sum_{\gamma} c_{(k-1)\gamma}(t,y)\Delta x^{\gamma}=\sum_{\gamma,l}c_{(k-1)\gamma l}(y)\Delta x^{\gamma}t^l,
\end{equation}
we have
\begin{equation}
\begin{array}{ll}
 c_{k} (t,x,y) =\sum_{\gamma,l}lc_{(k-1)\gamma l}(y)\Delta x^{\gamma}t^l\\
\\
+\sum_{\gamma}{\big \{}\sum_i\sum_{\beta + \alpha = \gamma} (\beta_i +1)(\alpha_i +1 )c_{r(\beta +1_i)}(t,y)c_{(k-1-r)(\alpha +1_i)}(t,y)   \\
\\
 +\sum_i (\gamma_i+2)(\gamma_i+1)c_{k(\gamma+2_i)}+ 
\sum_{\beta + \alpha = \gamma}(\sum\frac{1}{\beta !}b_{m,\beta}(t,y)\times\\
\\
(\alpha_i+1)c_{(k-1)(\alpha+1_i)}{\big \}}\left( \sum_{\delta=0}^{\gamma}p_{k\delta}^{y\gamma} \Delta x^{\delta}\right),  
\end{array}
\end{equation}
where the $p_{k\delta}^{y\gamma}$ are defined exactly as before. 
The proof of convergence is then analogue to the time-inhomogeneous case.

\section{Generalisation in the case of variable coefficients }

In order to provide a first discussion to various types of semi-elliptic equations let us formally compute the local expansion 
\begin{equation}\label{b}
p(\delta t,x,y)=\frac{1}{\sqrt{2\pi t}^n}\exp\left(-\frac{d^2(x,y)}{2 t}+\sum_{k= 0}^{\infty}c_k(x,y) t^k\right), 
\end{equation}
of the  parabolic equation
\begin{equation}\label{equationaaa1}
\frac{\partial u}{\partial  t}-\frac{1}{2}\sum_{ij}a^*_{ij}(x)\frac{\partial^2 u}{\partial x_i\partial x_j}-\sum_i b_i(x)\frac{\partial u}{\partial x_i}=0:=Lu
\end{equation}

First we compute for $ t>0$
\begin{equation}
\begin{array}{llll}
\hspace{0.4cm}\frac{\partial p}{\partial  t}=\left( -\frac{n}{2 t}+\frac{d^2(x,y)}{2 t^2}+\sum_{k=0}^{\infty}(k+1)c_{k+1}(x,y) t^k  \right)p, \\
\\
\hspace{0.4cm}\frac{\partial p}{\partial x_i}=\left( -\frac{d^2_{x_i}}{2t}+\sum_{k=0}^{\infty}\frac{\partial c_{k}}{\partial x_i}(x,y)t^k  \right)p,\\ 
\\
\frac{\partial^2 p}{\partial x_i \partial x_j}=\Big( \left( -\frac{d^2_{x_i}}{2t}+\sum_{k=0}^{\infty}\frac{\partial c_{k}}{\partial x_i}(x,y)t^k  \right)
\left( -\frac{d^2_{x_j}}{2t}+\sum_{k=0}^{\infty}\frac{\partial c_{k}}{\partial x_j}(x,y)t^k\right) \\
\\  
\hspace{2cm}-\frac{d^2_{x_ix_j}}{2t}+\sum_{k=0}^{\infty}\frac{\partial^2 c_k}{\partial x_i\partial x_j}t^k \Big) p.
\end{array}
\end{equation}
Plugging into \eqref{equationaaa1} we get 
\begin{equation}
\begin{array}{ll}
\Big(-\frac{n}{2t}+\frac{d^2(x,y)}{2t^2}+\sum_{k=0}^{\infty}(k+1)c_{k+1}(x,y)t^k
-\frac{1}{2}\sum_{ij} a^*_{ij}(x)\times \\
\\
\Big( \left( -\frac{d^2_{x_i}}{2t}+\sum_{k=0}^{\infty}\frac{\partial c_{k}}{\partial x_i}(x,y)t^k  \right)\left( -\frac{d^2_{x_j}}{2t}+\sum_{k=0}^{\infty}\frac{\partial c_{k}}{\partial x_j}(x,y)t^k\right)
-\frac{d^2_{x_ix_j}}{2t}\\
\\
+\sum_{k=0}^{\infty}\frac{\partial^2 c_k}{\partial x_i\partial x_j}t^k \Big)\Big)
-\sum_i b_i(x)\left( -\frac{d^2_{x_i}}{2t}+\sum_{k=0}^{\infty}\frac{\partial c_{k}}{\partial x_i}(x,y)t^k  \right)\Big)p=\\
\end{array}
\end{equation}
\begin{equation}
\begin{array}{ll}
\Big(-\frac{n}{2t}+\frac{d^2(x,y)}{2t^2}+\sum_{k=0}^{\infty}(k+1)c_{k+1}(x,y)t^k\\
\\
-\frac{1}{2}\sum_{ij} a^*_{ij}(x)\Bigg[ \frac{d^2_{x_i}}{2t}\frac{d^2_{x_j}}{2t}
-\frac{d^2_{x_i}}{2t}\left( \sum_{k=0}^{\infty}\frac{\partial c_k}{\partial x_j}t^k\right)
-\frac{d^2_{x_j}}{2t}\left( \sum_{k=0}^{\infty}\frac{\partial c_k}{\partial x_i}t^k\right)\\
\\
+\sum_{k=0}^{\infty}\left(\sum_{l=0}^{k}\frac{\partial c_l}{\partial x_i} \frac{\partial c_{k-l}}{\partial x_j}\right)t^k
 
-\frac{d^2_{x_ix_j}}{2t}+\sum_{k=0}^{\infty}\frac{\partial^2 c_k}{\partial x_i\partial x_j}t^k \Bigg]\\
\\ 
-\sum_i b_i(x)\left( -\frac{d^2_{x_i}}{2t}+\sum_{k=0}^{\infty}\frac{\partial c_{k}}{\partial x_i}(x,y)t^k  \right)\Big)p=
0 .
\end{array}
\end{equation}
Collecting terms of order $t^{-2}$ we have
\begin{equation}\label{ed}
d^2=\frac{1}{4}\sum_{ij}d^2_{x_i}a^*_{ij}d^2_{x_j}.
\end{equation}
Note that here $d^2_{x_i}$ is the partial derivative of $d^2$ with respect to $x_i$.
Equation (\ref{ed}) is closely connected to a Hamilton-Jacobi equation and admits a unique solution if the boundary condition, i.e. the condition $d(x,y)=0$ if $x=y$, is satisfied. 
Collecting terms of order $t^{-1}$ we get
\begin{equation}\label{1ga}
-\frac{n}{2}+\frac{1}{2}Ld^2+\frac{1}{2}\sum_{ij} a^*_{ij}(x)\Big(\frac{d^2_{x_j}}{2}\frac{\partial c_{0}}{\partial x_i}(x,y)+\frac{d^2_{x_i}}{2}\frac{\partial c_{0}}{\partial x_j}(x,y) \Big)=0.
\end{equation}
Equation \eqref{1ga} is a linear first order equation which can be written as 
\begin{equation}\label{c01e}
-\frac{n}{2}+\frac{1}{2}Ld^2+\frac{1}{2}\sum_{i} \left( \sum_j\left( a^*_{ij}(x)+a^*_{ji}(x)\right) \frac{d^2_{x_j}}{2}\right) \frac{\partial c_{0}}{\partial x_i}(x,y)=0.
\end{equation}
As we shall see later, this equation together with some boundary condition 
\begin{equation}\label{c01b}
c_0(y,y)=-\frac{1}{2}\ln \sqrt{\mbox{det}\left(a^{*ij}(y) \right) }
\end{equation}
determines $c_0$ uniquely for each $y\in {\mathbb R}^n$. In general, the boundary condition on $c_0$ ensures that $p$ is a probability density (i.e. integrates to 1). This is not essential as far as existence, uniqueness, and convergence of the coefficient functions $c_k$ are concerned. If we define $c_0(x,y)-c_0(y,y)=:\overline{c}_0(x,y)$ then $\overline{c}_0$ satisfies the equation \eqref{c01e} too with the boundary condition $c_0(x,y)=0$ if $x=y$.
For $k+1\geq 1$ we get
\begin{equation}\label{1gaa}
\begin{array}{ll}
(k+1)c_{k+1}(x,y)+\frac{1}{2}\sum_{ij} a^*_{ij}(x)\Big(
\frac{d^2_{x_i}}{2}\frac{\partial c_{k+1}}{\partial x_j}
+\frac{d^2_{x_j}}{2} \frac{\partial c_{k+1}}{\partial x_i}\Big)\\
\\
=\frac{1}{2}\sum_{ij}a^*_{ij}(x)\sum_{l=0}^{k}\frac{\partial c_l}{\partial x_i} \frac{\partial c_{k-l}}{\partial x_j}
+\frac{1}{2}\sum_{ij}a^*_{ij}(x)\frac{\partial^2 c_k}{\partial x_i\partial x_j}  
+\sum_i b_i(x)\frac{\partial c_{k}}{\partial x_i},
\end{array}
\end{equation}
with the boundary conditions
\begin{equation}\label{Rk}
c_{k+1}(x,y)=R_k(y,y) \mbox{ if }~~x=y,
\end{equation}
$R_k$ being the right side of \eqref{1gaa}.  

We have  
\begin{thm}
If the assumptions (\ref{coeffa}) and (\ref{coeffb}) are satisfied, then there exists a finite time horizon $T_0$ such that on the domain $\Omega\times (0,T_0]$ for any finite $T_0>0$ and any domain $\Omega\subseteq {\mathbb R}^n$ a constant $\beta$ can be computed such that the fundamental solution of 
\begin{equation}\label{parasystthm}
\frac{\partial u}{\partial t}=\sum_{i,j=1}^n a^*_{ij}\frac{\partial^2 u}{\partial x_j^2} 
+\sum_{i=1}^n b_{i}\frac{\partial u}{\partial x_i}
\end{equation}
has the pointwise valid representation
\begin{equation}\label{pj}
p(t,x,0,y)=\frac{1}{\sqrt{4\pi t(\tau)}^n}\exp\left(-\frac{\sum_{i=1}^n \Delta x_i^2}{4t}\right)\left(\sum_{k=0}^{\infty}d_{k}(t,x,y)t^k \right),  
\end{equation}
for $j=1,\cdots ,n$, and for $(t,x) \in(0,\beta T_0)\times \Omega$. If (\ref{finfou}) and (\ref{finfoua}) hold then a lower bound of the constant $\beta$ is given by
\begin{equation}\label{beta0}
\beta <\frac{1}{6(n(2|m_0|+1))\overline{e}R^2|m_0|^{2}},
\end{equation}
where $2|m_0|+1$ is (an upper bound of) the number of terms in the finite Fourier representation of $b_i$ (any $i\in \left\lbrace 1,\cdots ,n\right\rbrace$ along with $|m_0|:=\max_{j\in \left\lbrace 0,\cdots n\right\rbrace}m^j_0$ and $R$ is a radius of a ball $B_R(0)$ such that the spatial part of the domain $\Omega$ is included, i.e. $B_R(0)\supseteq \Omega$, and $\overline{e}$ is now an upper bound for the Fourier coefficients of the Fourier representation of the drift function $b_i$ where $i\in \left\lbrace 1,\cdots ,n\right\rbrace$ and of the Fourier coefficients of the Fourier representation of the diffusion functions $a_{ij},~1\leq i,j\leq n$.  
For the coefficient functions $d_k$ the following holds: for $k=0$ we have
\begin{equation}\label{c0}
 d_{0}(x,y)= \exp\left( c_0(x,y)\right) ,
\end{equation}
\begin{equation}
d_m(t,x,y)=\sum_{k=1}^m \frac{k}{m}d_{m-k}c_k,
\end{equation}
where $c_k$ is defined as above.
\end{thm}

\section{Application to linear semi-elliptic equations}

First we consider the application to (micro)-hypoelliptic equations. Then we consider the application to semi-elliptic equations which are (micro)-hypoelliptic on a linear subspace. More details about the analysis of semi-linear equations will be provided in version 3 of \cite{FKSE} (which will appear in arXiv shortly).

\subsection{Application to (micro)-hypoelliptic equations}

In the situation of linear semi-elliptic equations described in the introduction consider the case $n=d$, i.e. consider a matrix-valued function $x\rightarrow (a_{ji})^{d,m}(x),~0\leq i,j\leq n$ on ${\mathbb R}^n$, and $m+1$ smooth vector fields of dimension $d$ 
\begin{equation}
A_i=\sum_{j=1}^d a_{ji}\frac{\partial}{\partial x_j},~1\leq j\leq m,
\end{equation}
where $0\leq i\leq m$.
 Consider the 
Cauchy problem on $[0,T]\times {\mathbb R}^d$ with time horizon $T>0$:
\begin{equation}
	\label{projectiveHoermanderSystem}
	\left\lbrace \begin{array}{ll}
		\frac{\partial u}{\partial t}=\frac{1}{2}\sum_{i=1}^mA_i^2u+A_0u\\
		\\
		u(0,x)=f(x).
	\end{array}\right.
\end{equation}
Assume that (\ref{projectiveHoermanderSystem}) satisfies the H\"{o}rmander condition with respect to the subspace ${\mathbb R}^d$, i.e. assume that
\begin{equation}\label{Hoer*}
\left\lbrace A_i, \left[A_j,A_k \right], \left[ \left[A_j,A_k \right], A_l\right],\cdots |1\leq i\leq m,~0\leq j,k,l\cdots \leq m \right\rbrace 
\end{equation}
spans ${\mathbb R}^d$ at each point $x$. The existence of regular solutions of the Cauchy problem in this case is well known (cf. \cite{H} ). Indeed H\"{o}rmander's result shows us that there exists a family of smooth transition densities if (\ref{Hoer*}) holds for every $x\in {\mathbb R}^d$ . 
Indeed we may extract the following Malliavin estimate of the density from \cite{S}.
\begin{thm}\label{stroock}
Consider an $n$-dimensional diffusion process associated to (\ref{projectiveHoermanderSystem}) of the form
\begin{equation}
dX_t=\sum_{i=1}^nb_i(X_t)dt+\sum_{j=1}^{d}\sigma_{ij}(X_t)dW^j_t
\end{equation}
with $X(0)=x\in {\mathbb R}^n$ with values in ${\mathbb R}^n$ and on a time interval $[0,T]$,
i.e. assume that the solution of the Cauchy problem (\ref{projectiveHoermanderSystem}) has the probabilistic representation
\begin{equation}
 u(t,x)=E^x\left(f(X_t)\right) 
\end{equation}

Assume that $b_i,\sigma_{ij}\in C^{\infty}_{lb}$. Then the law of the process $X$ is absolutely continuous with respect to the Lebesgue measure and the density $p$ exists and is smooth, i.e. 
\begin{equation}
\begin{array}{ll}
p:(0,T]\times {\mathbb R}^n\times{\mathbb R}^n\rightarrow {\mathbb R}\in C^{\infty}\left( (0,T]\times {\mathbb R}^n\times{\mathbb R}^n\right). 
\end{array}
\end{equation}
Moreover, for each nonnegative natural number $j$, and multiindices $\alpha,\beta$ there are increasing functions of time
\begin{equation}\label{constAB}
A_{j,\alpha,\beta}, B_{j,\alpha,\beta}:[0,T]\rightarrow {\mathbb R},
\end{equation}
and functions
\begin{equation}\label{constmn}
n_{j,\alpha,\beta}, 
m_{j,\alpha,\beta}:
{\mathbb N}\times {\mathbb N}^n\times {\mathbb N}^n\rightarrow {\mathbb N},
\end{equation}
such that 
\begin{equation}\label{KEest}
|D^{j}_t D^{\alpha}_xD^{\beta}_yp(t,x,y)|\leq \frac{A_{j,\alpha,\beta}(1+x)^{m_{j,\alpha,\beta}}}{t^{n_{j,\alpha,\beta}}}\exp\left(-B_{j,\alpha,\beta}(t)\frac{(x-y)^2}{t}\right) 
\end{equation}
Moreover, all functions (\ref{constAB}) and  (\ref{constmn}) depend on the level of iteration of Lie-bracket iteration at which the H\"{o}rmander condition becomes true.
\end{thm}
Now consider the recursion equations (\ref{1ga}) and (\ref{1gaa}) for $c_0$ and $c_{k+1}$ for $k\geq 0$ of the expansion (\ref{b}) of the last section.

Let $f:{\mathbb R}^d\rightarrow {\mathbb R}$ be a continuous function of atmost exponential growth, i.e. 
 $$\mbox{for all $x \in \mathbb{R}^{n}$}
		 ~|f(x)|\leq C\exp(C|x|)~~\mbox{ for some constant} C>0.$$
Assume that $a_{ij}, b_i$ are bounded coefficients with bounded derivatives with
 \begin{equation}
 |D^{\alpha}_xa_{ij}|\leq C^{|\alpha|},~|D^{\beta}_xa_{ij}|\leq C^{|\beta|}
 \end{equation}
for some $C$ (e.g. $a_{ij}$ and $b_i$ have representations in form of finite Fourier series). 
Then under these conditions using the \cite{KSest} one may prove that
\begin{equation}\label{weakcon}
 \lim_{\epsilon\downarrow 0}\int_{{\mathbb R}^d}
 f(y)p_{D^{\epsilon}_{\mbox{{\tiny Deg}}}}(t,x,y)dy=
 \int_{{\mathbb R}^d}f(y)p(t,x,y)dy,
\end{equation}
where $p$ denotes the density of the process $X$ according to theorem (\ref{stroock}) above.
Then we may approximate the density via expansion on $D^{\epsilon}_{\mbox{{\tiny Deg}}}$ where it is strictly elliptic and set up an higher order scheme. Details of analysis will be given in a subsequent paper. 
Note that this observation can be used to construct weak higher order estimates for diffusion market models which satisfy the H\"{o}rmander condition. Recall that an approximation scheme $Y$ converges weakly with order $\gamma >0$ to $X$ as $\Delta t\downarrow 0$ and with respect to a function class $C$, if for all $g\in C$
\begin{equation}\label{weakcon}
|E\left(g(X_T)\right)-E\left(g(Y_T)\right) |\leq C\Delta T^{\gamma}
\end{equation}
as $\Delta T\downarrow 0$. It depends on the regularity of functions in $C$ whether (\ref{weakcon}) is a strong condition. In finance we may low regularity of payoffs, e.g. in the case of digital payoffs. This leads to sophisticated weighted Monte-Carlo schemes of bounded variance. The related estimators in \cite{KKS} may be adapted to the more general situation in a quite straightforward way (well, the proof of bounded variance and the error estimates become a little more intricate).

\subsection{Applications to semi-linear equations which are hypoelliptic on some linear subspace}
Note that for $n>d$ the Cauchy problem (\ref{projectiveHoermanderSystem}) does not satisfy the H\"{o}rmander condition on the whole space ${\mathbb R}^n$ in general but only on a linear subspace. Especially, a density may not exist in a regular sense. As an example, consider the following Cauchy problem on $[0,T]\times{\mathbb R}^2$:
\begin{equation}
	\label{eq:simpleExampleReducedSystem}
	\left\lbrace \begin{array}{ll}
		\frac{\partial u}{\partial t}=\frac{1}{2}\sum_{i,j=1}^d(\sigma\sigma^T)_{ij}(x)\frac{\partial^2 u}{\partial x_i\partial x_j}+\sum_{i=d+1}^n\mu_i\frac{\partial u}{\partial x_i},\\
		\\
		u(0,x)=f(x_1,\cdots ,x_d)+g(x_{d+1},\cdots ,x_{n}).
	\end{array}\right.
\end{equation}
Assume that the H\"{o}rmander condition is satisfied on a subspace of dimension $d$ (for any $(x_{d+1},\cdots ,x_n)$ fixed). Hence the fundamental solution $p_d$ of
\begin{equation}
 \frac{\partial u}{\partial t}=\frac{1}{2}\sum_{i,j=1}^d(\sigma\sigma^T)_{ij}(x)\frac{\partial^2 u}{\partial x_i\partial x_j}+\sum_{i=d+1}^n\mu_i\frac{\partial u}{\partial x_i}
\end{equation}
exists (for any $(x_{d+1},\cdots ,x_n)$ fixed), and the solution of (\ref{eq:simpleExampleReducedSystem}) becomes
The solution of this equation is
\begin{equation}
\int_{{\mathbb R}^d}f(y_1,\cdots ,y_d)p_d(t,x,y)dy+g(x_{d+1}+\mu_{d+1} t,\cdots ,x_{n}+\mu_{n} t).
\end{equation}
This leads us to a `distributional density' of the form
\begin{equation}
p(t,x,y):=p_d(t,x,y)\delta(x_{d+1}+\mu_{d+1} t-y_{d+1},\cdots ,x_{n}+\mu_{n} t-y_n).
\end{equation}
We see from this example that a density exists only in a distributional sense.
However, if the initial data $f$ satisfy an exponential growth condition, and are smooth on the space ${\mathbb R}^n\setminus {\mathbb R}^d$, and locally $L^p$ on the subspace ${\mathbb R}^d$ where the H\"{o}rmander condition holds, then  the Cauchy problem (\ref{projectiveHoermanderSystem})
has a regular solution, i.e. a solution in $C^{\infty}\left( \left(0,T \right]\times{\mathbb R}^n\right)$.
The simple example shows that in general there exists no regular density in a situation of degenerate diffusion models with $n>d$. Especially, the regularity theory of densities in the context of Malliavin calculus does not apply directly (as we remarked in \cite{FKSE}). However, the regularity theory for densities of Malliavin calculus may still be useful, especially for the analysis on the subspace of dimension $d$ where the H\"{o}rmander condition holds.  Let us consider (\ref{projectiveHoermanderSystem}) from a probabilistic perspective. The associated Stratonovic integral of a process starting at $x\in {\mathbb R}^n$ is
\begin{equation}\label{strat}
X_t=x+\int_0^tA_0(X_s)ds +\sum_{k=1}^m \int_0^tA_k(X_s)\circ dW_k(s)
\end{equation}
where $W_i$ denote Brownian motions and $\circ$ indicates that the integral is in the Stratonovic sense. If $n=d$ and the H\"{o}rmander condition holds, then the associated covariance matrix process $\sigma_t$ is almost surely invertible, i.e. 
\begin{equation}
\sigma^{-1}_t\in L^p
\end{equation}
for every real number $p$ and $t$ in some arbitrary finite time horizon $[0,T]$. Clearly this is not true in the example above. Note, however, that the associated process exists, i.e. global existence for ordinary stochastic differential equations as is well-known in the context of elementary stochastic analysis. A standard theorem of ordinary stochastic differential equations (for statement and proof cf. \cite{O}) shows the existence of Levy-continuous solutions.  Next we state and an extension of a theorem in \cite{FKSE} (the proof will be given in the third version of \cite{FKSE} ). One can use estimates obtained in \cite{S} which we cited above. In order to see how analytic expansions can be used in this context, our main interest here are some of the constructive aspects of the scheme which leads to the global existence and regularity proof. Let's consider the theorem first. Our interst here is the use of analytic expansions in this context. A detailed proof will be given in version 3 of \cite{FKSE} in arXiv shortly. Let us consider the time-homogeneous case.
Consider a matrix function $x\rightarrow (v_{ji})^{n,m}(x),~1\leq j\leq n,~0\leq i\leq m$ on ${\mathbb R}^n$, and $m$ smooth vector fields 
\begin{equation}
V_i=\sum_{j=1}^n v_{ji}(x)\frac{\partial}{\partial x_j},
\end{equation}
where $0\leq i\leq m$.
Consider the 
Cauchy problem on $[0,T]\times {\mathbb R}^n$ (where $T>0$ is an arbitrary finite horizon)
\begin{equation}
	\label{projectiveHoermanderSystemgen}
	\left\lbrace \begin{array}{ll}
		\frac{\partial u}{\partial t}=\frac{1}{2}\sum_{i=1}^mV_i^2u+V_0u\\
		\\
		u(0,x)=f(x).
	\end{array}\right.
\end{equation}
This may be rewritten in the form
\begin{equation}
	\label{projectiveHoermanderSystemgen2}
	\left\lbrace \begin{array}{ll}
		\frac{\partial u}{\partial t}=\frac{1}{2}\sum_{i,j=1}^nv_{ij}^*(x)\frac{\partial^2}{\partial x_i\partial x_j}u+\sum_{j=1}^n v_{j0}(x)\frac{\partial u}{\partial x_j}\\
		\\
		u(0,x)=f(x),
	\end{array}\right.
\end{equation}
where
\begin{equation}
\left( v^*_{ij}\right) (x)=\sum_{k=1}^m\left(V_i\right)^{\otimes^2}.
\end{equation}
A general reduced Cauchy problem is defined by the condition that for all $t\in [0,T]$ and $x\in {\mathbb R}^n$
\begin{equation}
\left( v^*_{ij}\right) (x)
\end{equation}
has rank $d\equiv d(x)\leq n$, where for each $x\in {\mathbb R}^n$ the number $d(x)$ is determined by the  H\"{o}rmander condition at $x\in {\mathbb R}^n$, i.e.  
\begin{equation}\label{Hoergen}
\left\lbrace V_i(x), \left[V_j,V_k \right](t,x), \left[ \left[V_j,V_k \right], V_l\right](t,x),\cdots |1\leq i\leq m,~0\leq j,k,l\cdots \leq m \right\rbrace 
\end{equation}
spans a linear subspace $W_{x}$ of dimension $d(x)$.  In this case we may consider the intersection of all $x$-dependent subspaces which are spanned by the local H\"{o}rmander condition at $x$, i.e.
\begin{equation}
I_{H}:=\cap_{x\in  {\mathbb R}^n}W_{x}
\end{equation}
where $W_{x}$ is the vector subspace of dimension $d(x)$ spanned by (\ref{Hoergen}) above at $x$. Here the notation $I_{H}$ indicates that we consider a intersection of spaces defined by local H\"{o}rmander conditions. Our most general theorem will show that regular global solutions of (\ref{projectiveHoermanderSystemgen}) (rsp. (\ref{projectiveHoermanderSystemgen2})) exist if the data are rough (i.e. in $L^p\left({\mathbb R}^n\right)$ only in $I_H$ and are smooth on the complementary vector subspace ${\mathbb R}^n\setminus I_H$. In the following section we look at the situation from the perspective of elementary stochastic analysis. We observe that from this perspective regularity is closely linked to regularity of the data. We shall see then that a constructive analytic scheme together with Malliavin type estimates lead us to stronger results. 
Next consider a matrix-valued function $x\rightarrow (v_{ji})^{n,m}(x),~1\leq j\leq n,~0\leq i\leq m$ on ${\mathbb R}^n$, and $m$ smooth vector fields 
\begin{equation}
V_i=\sum_{j=1}^n v_{ji}(x)\frac{\partial}{\partial x_j},
\end{equation}
where $0\leq i\leq m$.

We have
\begin{thm}
	\label{mainthm}
	Let $1\leq p\leq \infty$. Consider the Cauchy problem (\ref{projectiveHoermanderSystemgen}) on $[0,T]\times {\mathbb R}^n$. Assume that the initial data function $f:{\mathbb R}^n\rightarrow {\mathbb R}$ satisfies
	\begin{equation}\label{payoff}
	\begin{array}{ll}
		(i) & ~\mbox{ the function } f\mbox{ is $L^p_{loc},1\leq p\leq \infty$ on }~I_H,\\
		\\
		(ii) & ~\mbox{ the function } f\mbox{ is $C^{\infty}$ on } {\mathbb R}^n\setminus I_H ,
		\\
		\\
		(iii) & ~\mbox{for all $x \in \mathbb{R}^{n}$}\\
		& ~|f(x)|\leq C\exp(C|x|)~~\mbox{ for some constant $C>0$}.
	\end{array}
\end{equation}
Assume that the coefficients are smooth (i.e. $C^{\infty}$) of linear growth with bounded derivatives, i.e.
\begin{equation}
	\label{smooth}
	v_{ji} \in C_{l,b}^{\infty}\left({\mathbb R}^n \right)
\end{equation}
for $i=0$ and $1\leq j\leq n$, or  $1\leq i\leq m$ and $1\leq j\leq n$.
Then the Cauchy problem (\ref{projectiveHoermanderSystemgen}) has a global classical solution $u$, where
\begin{equation}
u\in C^{\infty}\left(\left(0,T\right] \times {\mathbb R}^n \right), 
\end{equation}
where the singular behaviour is $t=0$ is determined by the Malliavin-type estimate in \cite{S} as follows:
for given natural numbers $m$ and $N$ there is a number $q$ such that the solution $u$ and its time derivatives up to order $m$ and its spatial derivatives up to order $N$ are located in the space. 
\begin{equation}
C^q_{m, N}\left([0,T]\times {\mathbb R}^n \right) :=\left\lbrace v| t^qv\in C_{m, N}\left([0,T]\times {\mathbb R}^n \right)\right\rbrace , 
\end{equation}
where
\begin{equation}
C_{m,N}\left([0,T]\times {\mathbb R}^n \right):=\left\lbrace f \ \vert \ \|f\|+\sum_{l\leq m}\|D^l_tf\|+\sum_{|\alpha|\leq N}\|D^{\alpha}_xf\|<\infty\right\rbrace ,
\end{equation}
and $\|.\|$ denoting the supremum norm. Moreover, $q=\max_{|\alpha|\leq N}n_{m,\alpha,{\bf 0}}-n/2$ where $n_{m,\alpha,{\bf 0}}$ is determined by the estimate in \cite{S} of the singular behavior of the density.

\end{thm}

Since we are interested in the application of analytic expansions let us consider the main steps of the proof which lead to the constructive scheme (cf. \cite{FKSE} for details). For most applications in finance the situation in (\ref{mainthm}) can be reduced to the block structure, where the whole space ${\mathbb R}^n$ can be decomposed into a part ${\mathbb R}^d$ where the H\"{o}rmander condition holds and a complementary part ${\mathbb R}^{n-d}$ where the operator looks like a vector field. For associated first order equations with source term one observes (for a proof cf. \cite{FKSE}) : 
\begin{prop}\label{prop}
Fix $x^d\in {\mathbb R}^d$. Assume that the conditions of theorem 1 are satisfied. Assume that $g\in C^1\left([0,T]\times {\mathbb R}^n \right) $. Then there exists a smooth global flow ${\cal F}^t$ generated by the vector field below on $[0,T]\times {\mathbb R}^{n-d}$ such that the first order equation problem
\begin{equation}\label{vecf}
\begin{array}{ll}
\frac{\partial u}{\partial t}=\sum_{i=d+1}^n\mu_i(x^d,x^{n-d})\frac{\partial}{\partial x_i}u+g(t,x^d,x^{n-d}),\\
\\
~~u(0,x^d,x^{n-d})=f(x^k,x^{n-d}),
\end{array}
\end{equation}
has the solution
\begin{equation}\label{sol}
u(t,x^d,x^{n-d})=f\left(x^d,{\cal F}^t x^{n-d}\right)+\int_0^tg(s,x^d,{\cal F}^{t-s}x^{n-d})ds. 
\end{equation}

\end{prop}
The notation above which indicates that some coordinates are fixed ($x^d$ or $x^{n-d}$) is a little cumbersome and we shall drop it sometimes writing just $x$ instead of $(x^d,x^{n-d})$ in the following when it is quite clear from the context which components of $x$ should be considered to be fixed.
Next we define a local iteration scheme involving global flows of the type discussed in Proposition \ref{prop} above  and solutions of parabolic equations of form
\begin{equation}\label{degdiff}
\frac{\partial u}{\partial t}=\sum_{i,j=1}^d a^*_{ij}(x)\frac{\partial^2 u}{\partial x_i\partial x_j}+\sum_{i=1}^n\mu_i(x)\frac{\partial u}{\partial x_i}
\end{equation}
with $x^{n-d}$ fixed. 
The natural ansatz is an AD-scheme of the following form: we define\newline
\textbf{Vector Field Step:} ($l\geq 0$)
\begin{equation}\label{degdiff3}
\begin{array}{ll}
&\frac{\partial u^{2l}}{\partial t}-\sum_{i=d+1}^n\mu_i(x)\frac{\partial u^{2l}}{\partial x_i}\\
\\
=&
\left\lbrace \begin{array}{ll}
\sum_{i,j=1}^d a^*_{ij}(x)\frac{\partial^2 u^{2l-1}}{\partial x_i\partial x_j}
+\sum_{i=1}^d\mu_i(x)\frac{\partial u^{2l-1}}{\partial x_i}~\mbox{if}~l\neq 0\\
\\
0~\mbox{if}~l=0.
\end{array}\right.

\end{array}
\end{equation}
and\newline
\textbf{Diffusion Step:} ($l\geq 1$)
\begin{equation}
	\label{degdiff2}
	\begin{array}{ll}
&\frac{\partial u^{2l-1}}{\partial t}-\sum_{i,j=1}^d a^*_{ij}(x)\frac{\partial^2 u^{2l-1}}{\partial x_i\partial x_j}-\sum_{i=1}^d\mu_i(x)\frac{\partial u^{2l-1}}{\partial x_i}\\
\\
=&\sum_{i=d+1}^n\mu_i(x)\frac{\partial u^{2l-2}}{\partial x_i}.
\end{array}
\end{equation}
For each $m$ we define $u^{m}(0,.)=f(.)$ and $u^{m+1}(0,.)=f(.)$.
Here, in equation (\ref{degdiff2}) we understand $(x_{d+1},\cdots,x_n)$ to be fixed, and in  (\ref{degdiff3}) we understand
$(x_{1},\cdots,x_d)$ to be fixed.
In order to prove convergence time step by time step we rewrite the scheme in time-dilatation coordinates ($\rho$ will be small)
\begin{equation}\label{transtime}
\begin{array}{ll}
t: [0,\infty)\rightarrow [0,\infty),\\
\\
t(\tau)=\rho \tau .
\end{array}
\end{equation}
Then we get an equivalent equation in $\tau$ where the coefficients of the symbol of the operator become small if $\rho$ is small. We have
$$\frac{dt}{d\tau}=\rho ,$$
and
\begin{equation}
	\label{tauredpara}
	\left\lbrace  \begin{array}{ll}
		\frac{\partial u}{\partial \tau}-\rho\sum_{i,j=1}^{d}a^*_{ij}(x)\frac
		{\partial^{2}u}{\partial x^{i}\partial x^{j}}-\rho\sum_{i=1}^{n}
		\mu_i (x)\frac{\partial u}{\partial x^{i}}=0,\\
		\\
	u(0,x)=f(x).
	\end{array}\right.
\end{equation}

An iteration step of the scheme considered in transformed time $\tau$ for some time horizon $[0,T_0]$ is then given by
\begin{equation}
\begin{array}{ll}
&\frac{\partial u^{\rho ,2l}}{\partial \tau}-\sum_{i=d+1}^n\rho\mu_i(x)\frac{\partial u^{\rho ,2l}}{\partial x_i}\\
\\
=&\sum_{i,j=1}^d \rho a^*_{ij}(x)\frac{\partial^2 u^{\rho ,2l-1}}{\partial x_i\partial x_j}
+\sum_{i=1}^d\rho \mu_i(x)\frac{\partial u^{\rho ,2l-1}}{\partial x_i},
\end{array}
\end{equation}
 and 
\begin{equation}
\begin{array}{ll}
&\frac{\partial u^{\rho ,2l-1}}{\partial \tau}-\rho\sum_{i,j=1}^d a^*_{ij}(x)\frac{\partial^2 u^{\rho ,2l-1}}{\partial x_i\partial x_j}-\sum_{i=1}^k\rho\mu_i(x)\frac{\partial u^{\rho ,2l-1}}{\partial x_i}\\
\\
=&\sum_{i=d+1}^n\rho\mu_i(x)\frac{\partial u^{\rho ,2l-2}}{\partial x_i},
\end{array}
\end{equation}
for $l\geq 1$.
We start the scheme with
\begin{equation}
\begin{array}{ll}
\frac{\partial u^{\rho ,0}}{\partial \tau}-\sum_{i=d+1}^n\rho \mu_i(x)\frac{\partial u^{\rho ,0}}{\partial x_i}=0.
\end{array}
\end{equation}
The initial conditions are
\begin{equation}
u^{\rho ,m}(0,x)=f(x),~~m\geq 0,
\end{equation}
where for $m=1,3,\cdots$ $(x_{d+1},\cdots,x_n)$ is fixed, and for $m=0,2,\cdots$
$(x_{1},\cdots,x_d)$ is fixed. The solution can be constructed in the form in the form
\begin{equation}
u^{\rho}(\tau,x)=u^{\rho, 1}(\tau ,x)+\sum_{l\geq 1}\delta u^{\rho ,2l+1}(\tau ,x),
\end{equation}
where for $l\geq 1$
\begin{equation}
\delta u^{\rho ,2l+1}=u^{\rho ,2l+1}-u^{\rho ,2l-1}
\end{equation}
satisfies
\begin{equation}\label{parabdelta}
\begin{array}{ll}
&\frac{\partial \delta u^{\rho ,2l+1}}{\partial \tau}-\rho\sum_{i,j=1}^d a_{ij}(x)\frac{\partial^2 \delta u^{\rho ,2l+1}}{\partial x_i\partial x_j}-\sum_{i=1}^d\rho\mu_i(t(\tau),x)\frac{\partial \delta u^{\rho ,2l+1}}{\partial x_i}\\
\\
=&\sum_{i=d+1}^n\rho\mu_i(t(\tau),x)\frac{\partial \delta u^{\rho ,2l}}{\partial x_i},
\end{array}
\end{equation}
and in each substep where the right side in (\ref{parabdelta})
\begin{equation}
\delta u^{\rho ,2l}=u^{\rho ,2l}-u^{\rho ,2l-2}
\end{equation}
satisfies
\begin{equation}\label{rec1}
\begin{array}{ll}
&\frac{\partial \delta u^{\rho ,2l}}{\partial \tau}-\sum_{i=d+1}^n\rho\mu_i(t,x)\frac{\partial \delta u^{\rho ,2l}}{\partial x_i}\\
\\
=&\sum_{i,j=1}^d \rho a^*_{ij}(x)\frac{\partial^2 \delta u^{\rho ,2l-1}}{\partial x_i\partial x_j}
+\sum_{i=1}^d\rho\mu_i(x)\frac{\partial \delta u^{\rho ,2l-1}}{\partial x_i}.
\end{array}
\end{equation}
Moreover, for $m\geq 1$ $\delta u^{\rho ,m}$ has zero initial conditions, i.e. $\delta u^{\rho ,m}(0,x)=0$.
For small $\rho$ the scheme just described is locally convergent with respect to time. Then iteration of the scheme in time using the semigroup property leads to  a convergent scheme of a global solution to the Cauchy problem 
\begin{equation}
	\label{redparatrans}
	\left\lbrace  \begin{array}{ll}
		\frac{\partial u^{\rho}}{\partial \tau}-\frac{1}{2}\sum_{i,j=1}^{d}\rho a^*_{ij}(x)\frac
		{\partial^{2}u^{\rho}}{\partial x^{i}\partial x^{j}}-\sum_{i=1}^{n}
		\rho\mu_i (x)\frac{\partial u^{\rho}}{\partial x^{i}}=0 \text{,}\\
		\\
		u^{\rho}(0,x)=f(x).
	\end{array}\right.
\end{equation}
Note that iteration in time means that we start the next time step with the initial data $u^{\rho}(T_0,.)$, and after repeating the scheme above we get the next initial data $u^{\rho}(2T_0,.)$ and so on. The choice of $\rho$ depends on certain a priori estimates in\cite{S}. At each time step approximations of the densities of the diffusion substeps can be constructed according to the preceeding section. This leads to efficient schemes for a considerable class of semi-elliptic equations. For problems in higher dimensions probabilistic weighted Monte-Carlo schemes are constructed from the scheme above in a natural way (we shall discuss them in version 3 of \cite{FKSE} shortly). 

\section{Application to American derivatives}

We reconsider the front-fixing method in the case of a multivariate put option.
We extend the considerations of \cite{K} to (micro)-hypoelliptic operators. Our main interest here is how analytic expansions may be used in this context. Further ananlysis as well as details of global existence and regularity proofs can be found in a subsequent paper. We start with the operator
\begin{equation}\label{orig}
\frac{\partial u}{\partial t}+Lu\equiv\frac{\partial u}{\partial t}+\frac{1}{2}\sum_{ij}v_{ij}S_iS_j\frac{\partial^ 2 u}{\partial S_i\partial S_j}+r\left(\sum_i S_i\frac{\partial u}{\partial S_i}-u \right), 
\end{equation}
where $v_{ij}=(\sigma\sigma^T)_{ij}$ and $r$ may depend on time $t$ and spatial variables $S$. We do not assume that $v_{ij}$ the volatility matrix is strictly elliptic, but we assume that the operator is (micro)-hypoelliptic in the continuation region, or that the H\"{o}rmander condition holds in the continuation region. This may be rephrased nicely in the context of the frontfixing method on a half-space.
Let ${\cal E}\subset [0,T]\times {\mathbb R}^n_+$ denote the exercise region and for each $t\in [0,T]$ let ${\cal E}_t$ denote the $t$-section of the exercise region, i.e. ${\cal E}_t:=\left\lbrace x|(t,S)\in {\cal E}\right\rbrace$. 
In general beside basic standard assumptions on diffusion market models introduced in the previous Section we shall assume that 

\begin{itemize}
\item[(GG)]
for each $t\in [0,T]$ we assume that
$0\in {\cal E}_t$ and that ${\cal E}_t$ is star-shaped with respect to $0$, i.e. for all $S\in {\cal E}_t$ and all $\lambda\in [0,1]$ we assume that $\lambda S\in {\cal E}_t$. 
\begin{rem}
Note that this means that for a fixed "angle" at $S=(S_1,\cdots ,S_n)$, i.e. at
$$
\phi_S:=\left( \frac{S_2}{\sum_{i=1}^n S_i},\cdots ,\frac{S_n}{\sum_{i=1}^n S_i}\right) 
$$
we have one intersection point of the free boundary of the section ${\cal E}_t$ and  the ray through $0$ which is determined by the angle $\phi_S$.
\end{rem}
\end{itemize}
The condition (GG) is called the global graph condition. The condition (GG) holds if $x\rightarrow u(t,x)$ is convex, where $(t,x)\rightarrow u(t,x)$ denotes the value function of an American Put. This is a sufficient (not necessary) condition for (GG) to hold. Especially, this condition is satisfied for the multivariate Black-Scholes model (Consider the Snell envelope definition in order to verify convexity). However, note that
from the Snell envelope representation of the price of an American index Put 
 \begin{equation}
 p_A(t,S;K)=\sup_{\tau \in \mbox{Stop}_{[0,T]}}E_Q\left(K-\sum_{j=1}^nS_j\right) 
 \end{equation}
we see that the price function $p_A$ of an American put option with strike $K$ and with maturity $T>0$
\begin{equation}
(t,S;K)\rightarrow p_A(t,S;K)
\end{equation}
is homogenous of degree $1$ with respect to $(S,K)$, i.e. for any $\lambda >0$ we have for all $t\in [0,T]$
\begin{equation}
p_A(t,\lambda S;\lambda K)=\lambda p_A(t, S; K).
\end{equation}
Now we can prove
\begin{prop}
Assume the American put is written in a market where the no-arbitrage condition holds.
The price function $p_A$ of an American put option with strike $K$ and with maturity $T>0$
has the property that the function
\begin{equation}
S\rightarrow p_A(t,S;K)
\end{equation}
is convex for all $t\in [0,T]$ and $K$ fixed. Hence, (GG) is satisfied in this case independent of the underlying model.
\end{prop}

\begin{proof}
Comparing a portfolio $\Pi_A(t)$ consisting of  $\mu$ American puts with strike $K_1$ and $(1-\mu)$ American Puts with strike $K_2$ with a portfolio $P_B$ consisting of one American Put with strike $K=\mu K_1+(1-\mu) K_2$ we get
\begin{equation}
p_A(t,S;K)\leq \mu p_A(t,S;K_1)+(1-\mu)p_A(t,S;K_2).
\end{equation}
Then for fixed $K$ we may write convexity of $p_A$ in $S=(\mu\lambda_1+(1-\mu \lambda_2) K=\mu S_1+(1-\mu) S_2$ in the form
\begin{equation}\label{conv}
\begin{array}{ll}
(\mu\lambda_1+(1-\mu \lambda_2)p_A\left( t,K;\frac{K}{(\mu\lambda_1+(1-\mu \lambda_2)}\right) \\
\\\leq \mu \lambda_1 p_A\left( t,K;\frac{K}{\lambda_1}\right)+(1-\mu)\lambda_2p_A\left( t,K;\frac{K}{\lambda_2}\right)
\end{array}
\end{equation}
Division of (\ref{conv}) by $(\mu\lambda_1+(1-\mu \lambda_2)$ and the observation
that $p_A$ on the left side of (\ref{conv}) may be rewritten in the form
\begin{equation}
p_A\left( t,K;\frac{K}{(\mu\lambda_1+(1-\mu \lambda_2)}\right)=
p_A\left( t,K;\rho\frac{K}{\lambda_1}+(1-\rho)\frac{K}{\lambda_2}\right)
\end{equation}
along with
\begin{equation}
\rho=\frac{\mu \lambda_1}{(\mu\lambda_1+(1-\mu \lambda_2)} 
\end{equation}
leads to the reduction of convexity with respect to the asset to convexity with respect to the strike.
\end{proof}
Furthermore, note homogeneity of order one in $(S,K)$ is a quite natural condition. If a selffinancing portfolio with initial value $\Pi_{t_0}$ reduplicates the payoff $(K-S_{\tau^*})$, then the portfolio $\lambda \Pi_t$ reduplicates the payoff $(K-S_T)$

Hence, the free boundary can be written in terms of the angles in form
\begin{equation}
(t,\phi_S)\rightarrow F(t,\phi_S). 
\end{equation}
We consider the transformation
\begin{equation}
\begin{array}{ll}
\psi: (0,T)\times {\mathbb R}^ n_+\rightarrow (0,T)\times[1,\infty)\times \left(0,1 \right)^ {n-1},\\
\\
\psi(t,S_1,\cdots ,S_n)= \left( t ,\frac{\sum_{i=1}^n S_i}{F},\frac{S_2}{\sum_{i=1}^n S_i},\cdots ,\frac{S_n}{\sum_{i=1}^n S_i}\right).
\end{array}
\end{equation}
Note that the spatial part of $\psi\left( (0,T)\times {\mathbb R}^ n_+\right) $ is homeomorph to the half space $H_{\geq 1}=\left\lbrace x\in {\mathbb R}^ n|x_1\geq 1 \right\rbrace $. In the following the domain $D$ is the interior of the image of $\psi$, i.e. $D:=(0,T)\times (1,\infty)\times \left(0,1 \right)^ {n-1}$. We have
\begin{equation}
S_1=x_1F\left(1-\sum_{j\geq 2}x_j\right),~S_j=x_jx_1F.
\end{equation}
We get
\begin{equation}\label{frontbd}
\left\lbrace \begin{array}{ll}
u_{t}=\frac{F_t}{F}x_1\frac{\partial u}{\partial x_1}+\frac{1}{2}\sum_{ij}a^F_{ij}\frac{\partial^ 2 u}{\partial x_i\partial x_j}+\sum_j b^F_j\frac{\partial u}{\partial x_j}+r\left( x_1\frac{\partial u}{\partial x_1}-u\right) ,\\
\\
(BC1)~~u(0,\infty,x_2,\cdots ,x_n)=0 \mbox{ on $x_1=\infty$}\\
\\
(BC2)~~u_{x_1}(t,1,x_2,\cdots ,x_n)-u(t,1,x_2,\cdots ,x_n)=-K \\
\hspace{6cm}\mbox{ on $x_1=1$}\\
\\
(BC3)~~F(t,x_2,\cdots ,x_n)=K-u(t,1,x_2,\cdots,x_n)\\
\\
(IC)~~u(0,x)=\max \{K-x_1, 0\}
\end{array}\right.
\end{equation}
\begin{rem}
We include (BC1) as an implicit boundary condition in order to indicate that \eqref{frontbd} is equivalent to an initial-boundary value problem of the second type on a finite domain (just by suitable additional transformation with respect to the variable $x_1$).
\end{rem}
The mixed condition $(BC2)$ follows from the smooth fit condition together with $(BC3)$ .
\begin{rem}
Note that in the context of market models based on Levy processes or, more generally, Feller processes the smooth fit condition does not hold in general and one has to be careful concerning generalization at this point. 
\end{rem}
In order to determine the coefficients $a^F_{ij}$ and $b^F_i$ we compute first
\begin{equation}
\begin{array}{ll}
F\frac{\partial x_j}{\partial S_i}=\frac{\delta_{ij}-x_j}{x_1},~
F\frac{\partial x_j}{\partial S_1}=1-\sum_{j\geq 2}(\delta_{ij}-x_j)\frac{F_j}{F},~
\frac{\partial}{\partial S_i}=\sum_j \frac{\partial x_j}{\partial S_i}\frac{\partial}{\partial x_j}.
\end{array}
\end{equation}
We observe that
\begin{equation}
\sum_i S_i\frac{\partial x_j}{\partial S_i}=\sum_i S_i\frac{\delta_{ij}-x_j}{x_1F}=0.
\end{equation}
It follows that
\begin{equation}
\begin{array}{ll}
\sum_i S_i\frac{\partial }{\partial S_i}&=\sum_{ij}\frac{\partial x_j}{\partial S_i}\frac{\partial }{\partial S_i}¬=\sum_{i}S_i\frac{\partial x_1}{\partial S_i}\frac{\partial }{\partial x_1}    +\sum_{j\geq 2}\left(\sum_i S_i\frac{\partial x_1}{\partial S_i} \right) \frac{\partial }{\partial x_j}     \\
\\
&=\sum_i S_i\frac{\partial x_1}{\partial S_i}\frac{\partial}{\partial x_1}=\sum_i S_i\left(\frac{1}{F}-\sum_{j\geq 2}\left(\delta_{ij}-x_j\right)\frac{F_j}{F^ 2} \right) \frac{\partial}{\partial x_1}\\
\\
&=x_1\frac{\partial }{\partial x_1}-\left(\sum_{i}S_i\right) \left(\sum_{j\geq 2}\left(\frac{\partial x_j}{\partial S_i}\right)\frac{F_j}{F}x_1 \right)\frac{\partial }{\partial x_1}=x_1\frac{\partial }{\partial x_1}.
\end{array}
\end{equation}
Hence, we have
\begin{equation}
r\left(\sum_i S_i \frac{\partial}{\partial S_i}\right)=rx_1\frac{\partial}{\partial x_1}.
\end{equation}
It is clear that
\begin{equation}
a^F_{ij}=\sum_{kl}v_{kl}S_kS_l\frac{\partial x_i}{\partial S_l}\frac{\partial x_l}{\partial S_k},~~
b^F_{j}=\sum_{kl}v_{kl}S_kS_l\frac{\partial^2 x_j}{\partial S_k\partial S_l}.
\end{equation}
In order to determine the latter coefficient functions we compute
\begin{equation}
\frac{\partial x_j}{\partial S_i}=\frac{1}{F}\frac{\delta_{ij}-x_j}{x_1}, j\geq 2,~~
\frac{\partial x_1}{\partial S_i}=\frac{1}{F}\left( 1-\sum_{j\geq 2}(\delta_{ij}-x_j)\frac{F_j}{F}\right).
\end{equation}
Next, for $j\geq 2$ we have
\begin{equation}
\frac{\partial^ 2 x_j}{\partial S_i\partial S_k}=\sum_l\frac{\partial \left(\frac{\delta_{ij}-x_j}{Fx_1} \right) }{\partial x_l}, \mbox{ and }
\end{equation}
\begin{equation}
\frac{\partial \left(\frac{\delta_{ij}-x_j}{Fx_1} \right)}{\partial x_l}
=-\frac{\delta_{jl}}{x_1F}+\frac{(x_j-\delta_{ij})(\delta_{1l}F+x_1F_l(1-\delta_{1l}))}{(x_1F)^ 2}.
\end{equation}
Finally,
\begin{equation}
\frac{\partial^ 2 x_1}{\partial S_i\partial S_k}=\sum_l \frac{\partial}{\partial x_l}\left(\frac{1}{F}-\sum_{j\geq 2}(\delta_{ij}-x_j)\frac{F_j}{F^ 2}\right) \frac{\partial x_l}{\partial S_k}, \mbox{ where}
\end{equation}
\begin{equation}
\begin{array}{ll}
\frac{\partial}{\partial x_l}\left(\frac{1}{F}-\sum_{j\geq 2}(\delta_{ij}-x_j)\frac{F_j}{F^ 2}\right)\\
\\
=-\frac{F_l}{F^2}(1-\delta_{1l})-\sum_{j\geq 2}\frac{-\delta_{jl}F_j+(\delta_{ij}-x_j)F_{jl}(1-\delta_{1l})-2F_jF_l(1-\delta_{1l})(\delta_{ij}-x_j)}{F^ 3}.
\end{array}
\end{equation}
Here $\delta_{ij}$ is always the Kronecker Delta, $F_j$ is short for $\frac{\partial F}{\partial x_j}$, $F_{jl}$ is short for $\frac{\partial^ 2 F}{\partial x_j\partial x_l}$. Now we have determined the explicit form of \eqref{frontbd}. The next step is to construct a representation of the solution of \eqref{frontbd} in terms of convolutions with the transition density, i.e the fundamental solution related to \eqref{frontbd}.

 Note that only the inner regularity of the free boundary function can be proved. 
Starting with some $u^0$ (solution of (\ref{frontbd}) for $F\equiv1$ for example) for numerical reasons we may consider an iteration scheme $v^n=tu^n$ with $v^0=tu^0$, and consider for $n\geq 1$ the following iteration for  \ref{frontbd}. Let
\begin{equation}\label{redpara*}
\begin{array}{ll}
v^n_{t}=\frac{F^{n-1}_t}{F^{n-1}}x_1 \frac{\partial v^n}{\partial x_1}+\frac{1}{2}\sum_{ij}a^{F^{n-1}}_{ij}\frac{\partial^ 2 v^n}{\partial x_i\partial x_j}+\sum_j b^{F^{n-1}}_j\frac{\partial v^n}{\partial x_j}\\
\\
+r\left( x_1\frac{\partial v^n}{\partial x_1}-v^n\right)+u^{n-1},
\end{array}
\end{equation}
where $v^n$ has zero initial condition. On the hyperplane $\left\lbrace x_1=1\right\rbrace $ we require
\begin{equation}
v^n_{x_1}(t,1,x_2,\cdots ,x_n)-v^n(t,1,x_2,\cdots ,x_n)=-t K.
\end{equation}
Given $u^{n-1}$ and $F^{n-1}$ for each $n$ we look first at the solution for $v^n$ in the form
(note that $\int_{D_b}$ below contains an integral $\int_0^t$)
\begin{equation}\label{itequation}
\begin{array}{ll}
v^n(t,x)=&\int_0^ t \int_H p_{v^n}(t,1, x;\tau,1, \hat{y}_1)\phi_{v^n} (\tau,1,\hat{y}_1)dH_yd\tau\\
\\
&+\int_{D_b} u^{n-1}(s,y)p_{v^n}(t,x;s,y)dyds,
\end{array}
\end{equation}
where $p_{v^n}$ is the fundamental solution of
\begin{equation}
v^n_{t}=\frac{1}{2}\sum_{ij}a^{F^{n-1}}_{ij}\frac{\partial^ 2 v^n}{\partial x_i\partial x_j}+\sum_j b^{F^{n-1}}_j\frac{\partial v^n}{\partial x_j}+\frac{F^{n-1}_t}{F^{n-1}}x_1 \frac{\partial v^n}{\partial x_1},
\end{equation}
and where $\phi_{v^n}$ solves an integral equation 
\begin{equation}\label{vol2}
\begin{array}{ll}
\frac{1}{2}\phi_{v_n}(t,1,\hat{x}_1)=\Gamma_{v_n}(t,1,\hat{x}_1)+\\
\\ \int_{t_0}^ t \int_{H_0} (\frac{\partial}{\partial x_1}p_{v_n}(t,1,\hat{x}_1,\tau,\hat{y}_1)-p_{v_n}(t,1,\hat{x}_1;\tau,1,\hat{y}_1))\phi_{v_n} (\tau,1,\hat{y}_1)dH_yd\tau,
\end{array}
\end{equation}
(defining $\phi_{v_n}$) with 
\begin{equation}
\begin{array}{rrr}
\Gamma_{v_n} (t,x)=&\int_{O} \left( \frac{\partial}{\partial x_1}p_{v_n}(t,x;t_0,y)- p_{v_n}(t,x;t_0,y)\right) t\psi_0(y)dy+t K.
\end{array}
\end{equation}
\begin{rem}
Asymptotic analysis shows that even $\sqrt{t}F_t$ is bounded as $t\downarrow 0$.
\end{rem}
For the next step we get $F^n$ and $u^n$ via $tF_n=K-v^n(1,\hat{x}_1)$ and $tu^n=v^n$. In the case of strictly elliptic operators $L$  we get convergence in adapted Banach spaces for $v^n$ (cf. \cite{K}). Regularity resuls and convergence of the scheme can be generalised to the case of micro-hypoelliptic operators $L$, i.e. if $L$ has the
representation
\begin{equation}
	\label{Lhypo}
	\begin{array}{ll}
		L\equiv\frac{1}{2}\sum_{i=1}^mA_i^2u+A_0u
	\end{array}
\end{equation}
with a matrix-valued function $x\rightarrow (a_{ji})^{d,m}(x),~0\leq i,j\leq n$ on ${\mathbb R}^n$, and $m+1$ smooth vector fields of dimension $d$ 
\begin{equation}
A_i=\sum_{j=1}^d a_{ji}\frac{\partial}{\partial x_j},~1\leq j\leq m,
\end{equation}
where $0\leq i\leq m$ such that the H\"{o}rmander condition with respect to the subspace ${\mathbb R}^d$, i.e. assume that
\begin{equation}\label{Hoer}
\left\lbrace A_i, \left[A_j,A_k \right], \left[ \left[A_j,A_k \right], A_l\right],\cdots |1\leq i\leq m,~0\leq j,k,l\cdots \leq m \right\rbrace 
\end{equation}
spans ${\mathbb R}^d$ at each point $x\in \left\lbrace  x|x\in {\mathbb R}^n \mbox{ and } x_1\geq 1\right\rbrace $. The proof uses the estimate in \cite{KS} cited above. Details will be given in \cite{Kb}.

\end{document}